\documentclass{article}
\usepackage[utf8]{inputenc}

\usepackage{amsmath,amssymb,amsthm}
\usepackage{bm}
\usepackage{algorithm,algorithmic}
\usepackage{enumerate}
\usepackage[margin=25mm]{geometry}
\usepackage[numbers]{natbib}
\usepackage{mathtools}
\mathtoolsset{showonlyrefs,showmanualtags}

\newcommand{\argmin}{\mathop{\rm argmin}\limits}
\newcommand{\dom}{\mathop{\rm dom}}

\newtheorem{theorem}{Theorem}
\newtheorem{corollary}{Corollary}
\newtheorem{lemma}{Lemma}

\newtheorem{example}{Example}
\newtheorem{proposition}{Proposition}

\title{Exact Penalization at D-Stationary Points of\\ Cardinality- or Rank-Constrained Problem}
\author{Shotaro Yagishita\thanks{The Institute of Statistical Mathematics, Japan, E-mail: syagi@ism.ac.jp}
\and
Jun-ya Gotoh\thanks{Department of Data Science for Business Innovation, Chuo University, Japan, E-mail: jgoto@indsys.chuo-u.ac.jp}}
\date{\today}

\begin{document}

\maketitle

\begin{abstract}
This paper studies the properties of d-stationary points of the trimmed lasso (Luo et al., 2013, Huang et al., 2015, and Gotoh et al., 2018) and the composite optimization problem with the truncated nuclear norm (Gao and Sun, 2010, and Zhang et al., 2012), which are known as tight relaxations of nonconvex optimization problems that have either cardinality or rank constraints, respectively.
First, we extend the trimmed lasso for broader applications and for conducting a unified analysis of the property of the generalized trimmed lasso.
Next, the equivalence between local optimality and d-stationarity of the generalized trimmed lasso is shown under a suitable assumption. 
More generally, the equivalence is shown for problems that minimize a function defined by the pointwise minimum of finitely many convex functions.
Then, we present new results of the exact penalty at d-stationary points of the generalized trimmed lasso and the problem with the truncated nuclear norm penalty under mild assumptions.
Our exact penalty results are not only new, but also intuitive, so that we can see what properties of the problems play a role for establishing the exact penalization. 
Lastly, matters relating to algorithms are also discussed.
\end{abstract}

\section{Introduction}\label{sec:intro}
Optimization problems that pursue sparsity or low-rank structures of the solutions frequently appear in practice.
One of the basic forms comes with the cardinality or rank constraint, which aims at directly constraining the number of non-zero elements of the vector of decision variables or the rank of the matrix, respectively. 
While such constrained problems have a broad range of applications, they are computationally intractable in general due to their nonconvexity and discontinuity.
One of the most popular ways to overcome the intractability is the relaxation of the cardinality and rank with the $\ell_1$ norm and nuclear norm, respectively \citep{tibshirani1996regression,jaggi2010simple}. 
In particular, the former approximation is known as lasso when it applies to least squares estimation. 
In order to reinforce the sparsity of solution, approximations of the $\ell_0$ norm such as the $\ell_p$-approximation with $0<p<1$, SCAD \citep{fan2001variable}, and MCP \citep{zhang2010nearly} have been proposed by employing nonconvex regularizers in place of the $\ell_1$ norm. 
However, their solution does not always satisfy the cardinality since all of those approaches are based on an approximation of the cardinality of the solution vector.

In recent years, nonconvex tight relaxations of the problem with the cardinality constraint and the one with the rank constraint have been proposed \citep{gao2010majorized,zhang2012matrix,luo2013new,huang2015two,gotoh2018dc}. 
In this paper, we focus on the tight relaxations using the {\it trimmed $\ell_1$ norm} and {\it truncated nuclear norm}, which are both nonconvex but continuous functions unlike the cardinality of vector and the rank of matrix. 
More specifically, we add those functions to the objective function instead of the cardinality constraint and rank constraint so that the trimmed $\ell_1$ norm and truncated nuclear norm play a role of a penalty function on the violation of the cardinality constraint and rank constraint, respectively. 
The former is called the trimmed lasso \citep{luo2013new,huang2015two,gotoh2018dc}.
In \citet{yagishita2024pursuit}, the trimmed lasso has been extended to be applicable to a structured sparsity.

Interestingly, it has been shown that under some conditions the penalty functions defined by the trimmed $\ell_1$ norm and truncated nuclear norm can be ``exact'' for the (structured) sparsity and rank constraints, respectively \citep{bi2016error,gotoh2018dc,ahn2017difference,bertsimas2017trimmed,liu2020exact,nakayama2021superiority,amir2021trimmed,yagishita2024pursuit,qian2023calmness,lu2023exact}.  
Here, by ``a penalty function is exact'' we mean that a certain kind of solution of the penalized problem satisfies the (structured) sparsity or rank constraint under suitable assumptions. 
It is noteworthy that the existing results of the exact penalization have been established under strong assumptions which do not seem to be confirmed easily in practice. 
For example, some statements are only valid at globally optimal solutions, but due to the nonconvexity of the relaxations, it is generally difficult to obtain a guaranteed global optimum in practice.
In that sense, the exact penalty result at globally optimal solutions is not practical. 
In addition, while there are other results valid at stationary points that are attainable by practical algorithms (e.g., those described in Section \ref{sec:discussions}), unfortunately, the problems to which the results can be applied are limited in practice.

To fill the gap between the assumption for exact penalization and practice, we first extend the trimmed lasso for broader applications and refer to it as the generalized trimmed lasso.
Since the generalized trimmed lasso includes all existing variations of the trimmed lasso, it enables us to analyze their properties in a unified manner.
As will be clear, some results are reinforced by the refinement of the analysis based on this generalization. 
Next, we derive a sufficient condition for a d-stationary point of the generalized trimmed lasso to be a locally optimal solution so as to associate the following results with local optimality. 
In addition, a condition is given for a more general class of optimization problems. 
Then, new results of exact penalization at d-stationary points of the generalized trimmed lasso and the penalized problem by the truncated nuclear norm are presented under mild assumptions. 
Our results are proven in a different way from the existing ones, leading to intuitive and stronger implications.
Finally, we discuss algorithms for finding a d-stationary point that satisfy the (structured) sparsity or rank constraint.

Contributions of our work are summarized as follows:
\begin{itemize}
\item 
For optimization problems with the trimmed $\ell_1$ norm, we clarify the class of optimization problems for which locally optimal solutions and d-stationary points are equivalent.
Within that class, any d-stationary point is proven to be locally optimal.
It is noteworthy that while using the information of the second order derivative of objective function is a common strategy, our result does not depend on the second order information and is established despite the inclusion of the nonsmooth regularizers.
\item We extend the trimmed lasso for broader applications and for conducting a unified analysis of the property. While the existing results of exact penalization for the context we are interested in have been established assuming different conditions in each context, the unified analysis leads us to an essential assumption for the trimmed $\ell_1$ norm to achieve an exact penalization at d-stationary points. 
Our analysis for the trimmed $\ell_1$ norm is extended to the truncated nuclear norm.
Our results can apply to the problems for which the exact penalization has not yet been established.
In particular, to the best of our knowledge, the exact penalization at d-stationary points for composite problems with the truncated nuclear norm is first presented in this paper.
For general nonlinear programming, the exact penalty results at stationary points were given by \citet{demyanov1998exact} and \citet{cui2021modern}.
However, due to the generality of nonlinear programming, it is not obvious to check whether the assumption in their results holds for problems we deal with.
In addition, the corresponding results of \citet{demyanov1998exact} and \citet{cui2021modern} assumed that the objective function is Lipschitz continuous, but we do not.
\end{itemize}

The rest of this paper is organized as follows.
The remainder of this section is devoted to notation and preliminary results, which include the equivalence result of local optimality and d-stationarity.
Section \ref{sec:GTL} introduces the generalized trimmed lasso and shows an exact penalty result at d-stationary points. 
Extensions to the exact penalty results for the trimmed lasso with some additional constraint and the problem with the truncated nuclear norm are shown in Sections \ref{sec:TL-const} and \ref{sec:truncated-nuclear}, respectively. 
In Section \ref{sec:discussions}, we discuss matters relating to algorithms for the penalized problems.
Finally, Section \ref{sec:conclusion} concludes the paper with some remarks.

\subsection{Notation and Preliminaries}
For an integer $n$, the set $[n]$ is defined by $[n]\coloneqq\{1,\ldots,n\}$.
Let $\mathbf{1}$ and $e_i$ denote the all-ones vector and the vector whose $i$th element is $1$ and 0 otherwise, respectively.
The standard inner product of $x,y\in\mathbb{R}^n$ is denoted by $x^\top y$.
The $\ell_1$ norm, $\ell_2$ norm, and $\ell_\infty$ norm of $x\in\mathbb{R}^n$ are defined by $\|x\|_1\coloneqq\sum_{i\in[n]}|x_i|$, $\|x\|_2\coloneqq\sqrt{x^\top x}$, and $\|x\|_\infty\coloneqq\max_{i\in[n]}|x_i|$, respectively.
For a vector $x\in\mathbb{R}^n$, $\|x\|_0$ denotes the number of nonzero elements.

Let $I$ denote the identity matrix of appropriate size.
By $\mathrm{diag}(x_1,\ldots,x_n)$ we denote the diagonal matrix whose $i$th diagonal element is $x_i$.
The standard inner product of matrices $X,Y\in\mathbb{R}^{m\times n}$ is defined by $X\bullet Y\coloneqq\mathrm{tr}(X^\top Y)$, where $\mathrm{tr}(Z)$ is the trace of a square matrix $Z$.
The Hadamard product is denoted by $X\circ Y$.
The Frobenius norm of $X\in\mathbb{R}^{m\times n}$ is defined by $\|X\|_F\coloneqq\sqrt{X\bullet X}$.
We denote the $i$th largest singular value of $X\in\mathbb{R}^{m\times n}$ by $\sigma_i(X)$  for $i=1,\ldots,\min\{m,n\}$. 
The smallest nonzero singular value of a nonzero matrix $X$ is denoted by $\sigma_{\min}(X)$.
For a matrix $X\in\mathbb{R}^{m\times n}$, we define $\|X\|_1\coloneqq\sum_{i\in[m]}\sum_{j\in[n]}|X_{i,j}|, \|X\|_2\coloneqq\sigma_1(X)$, and $\|X\|_*\coloneqq\sum_{i=1}^{\min\{m,n\}}\sigma_i(X)$.
Let $\mathrm{rank}(X)$ and $\ker(X)$ denote the rank and kernel of a matrix $X$, respectively.
The largest eigenvalue and smallest eigenvalue of a symmetric matrix $X$ are denoted by $\lambda_{\max}(X)$ and $\lambda_{\min}(X)$, respectively.
The set of symmetric matrices of order $n$ is denoted by $\mathcal{S}^n$.
For $X\in\mathcal{S}^n$, $X\succeq0$ means that $X$ is a positive semidefinite matrix.
For a linear subspace $\mathcal{L}$ of $\mathbb{R}^n$, $\mathcal{L}^\perp$ denotes its orthogonal complement.
The following lemma is immediately apparent by considering the singular value decomposition.

\begin{lemma}\label{lem:restricted bounded below}
Let $A\neq 0$ be an $m\times n$ matrix.
The inequality
\begin{equation}\label{eq:restricted bounded below}
\|Ax\|_2\ge\sigma_{\min}(A)\|x\|_2
\end{equation}
holds for $x\in\ker(A)^\perp$.
Additionally, the inequality is tight in the sense that, for any $\sigma>\sigma_{\min}(A)$, the one replacing $\sigma_{\min}(A)$ by $\sigma$ does not hold.
\end{lemma}

Let $\mathbb{E}$ be a finite-dimensional inner product space (in this paper, $\mathbb{E}$ stands for $\mathbb{R}^n$ or $\mathbb{R}^{m\times n}$).
For $\mathcal{C}\subset\mathbb{E}$ and $x\in\mathbb{E}$, $\mathcal{F}(x;\mathcal{C})$ is the feasible cone of $\mathcal{C}$ at $x$, that is,
\begin{align}
    \mathcal{F}(x;\mathcal{C})\coloneqq\begin{cases}
        \{d\in\mathbb{E}\mid x+\varsigma d\in\mathcal{C},~ \forall\varsigma\in(0,\varsigma')~ \mbox{for some}~ \varsigma'>0\}, &x\in\mathcal{C},\\
        \emptyset, &x\notin\mathcal{C}.
    \end{cases}
\end{align}
The domain of a function $f:\mathbb{E}\to(-\infty,\infty]$ is denoted by $\dom f\coloneqq\{x\in\mathbb{E}\mid f(x)<\infty\}$.
For $\mathcal{C}\subset\mathbb{E}$, $\delta_\mathcal{C}:\mathbb{E}\to\{0,\infty\}$ denotes the indicator function of $\mathcal{C}$.
A function $f:\mathbb{E}\to(-\infty,\infty]$ is said to be directionally differentiable if the directional derivative of $f$,
\begin{equation}
f'(x;d)\coloneqq\lim_{\varsigma\searrow0}\frac{f(x+\varsigma d)-f(x)}{\varsigma}
\end{equation}
exists for any point $x\in\dom f$ and direction $d\in\mathcal{F}(x;\dom f)$.
For a Lipschitz continuous directionally differentiable function, the following lemma holds, which is obvious and hence the proof is omitted.

\begin{lemma}\label{lem:d-derivative-lipschitz}
Let $f:\mathbb{E}\to(-\infty,\infty]$ is directionally differentiable and Lipschitz continuous on $\dom f$ with constant $M$ under some norm $\|\cdot\|$, that is, $|f(x)-f(y)|\le\|x-y\|$ holds for all $x,y\in\dom f$.
Then, the inequality
\begin{equation}
\left|f'(x;d)\right|\le M\|d\|
\end{equation}
holds for all $x\in\dom f$ and $d\in\mathcal{F}(x;\dom f)$.
\end{lemma}

The gradient of a function $f$ at a point $x$ is denoted by $\nabla f(x)$.
The vector $\nabla_{x_l} f(x)$ denotes the subvector of $\nabla f(x)$ corresponding to $x_l$, where $x_l$ is a subvector of $x$.
A continuous differentiable function $f$ on an open set $\mathcal{O}\subset\mathbb{E}$ is said to be $L$-smooth if there exists $L\ge0$ such that $\|\nabla f(x)-\nabla f(y)\|\leq L\|x-y\|$ for any $x, y\in\mathcal{O}$, where $\|\cdot\|$ is induced by the inner product of $\mathbb{E}$.

Besides, we present some results on the optimization problem:
\begin{align}\label{problem:general}
\underset{x\in \mathbb{E}}{\mbox{minimize}} & \quad F(x),
\end{align}
where $F:\mathbb{E}\to(-\infty,\infty]$ is directionally differentiable.
Formulation \eqref{problem:general} is a more general problem than the ones we will address in the next section and beyond.
We call $x^*\in\dom F$ a d(irectional)-stationary point of \eqref{problem:general} if $F'(x^*;d)\ge0$ holds for all $d\in\mathcal{F}(x^*;\dom F)$.
Obviously, any locally optimal solution of \eqref{problem:general} is a d-stationary point of the problem.
The converse argument holds for a certain class of functions.
We first give the following formula, which is an extension of Theorem 3.4 of \citet{delfour2019introduction}.

\begin{lemma}\label{lem:d-derivative-min}
Let $f_i:\mathbb{E}\to(-\infty,\infty]$ be directionally differentiable and $\dom f_i$ be closed convex for $i=1,\ldots,M$.
Suppose that $F(x)\coloneqq\min_{i\in[M]}f_i(x)$ is continuous on $\dom F$.
Then, $F$ is directionally differentiable and it holds that 
\begin{equation}
F'(x;d)=\min_{i\in I^*(x,d)}f_i'(x;d)
\end{equation}
for any $x\in\dom F$ and $d\in\mathcal{F}(x;\dom F)$, where $I^*(x,d)=\{i\in I(x,d)\mid f_i(x)=F(x)\}$ and $I(x,d)=\{i\in[M]\mid x\in\dom f_i,~ d\in\mathcal{F}(x;\dom f_i)\}$.
\end{lemma}

\begin{proof}
Let $x\in\dom F$ and $d\in\mathcal{F}(x;\dom F)$.
For any $i$ such that $x\in\dom f_i$ and $d\notin\mathcal{F}(x;\dom f_i)$, $f_i(x+\varsigma d)=\infty$ holds for all $\varsigma>0$ by the convexity of $\dom f_i$.
From the closedness of $\dom f_i$, there is a neighborhood $\mathcal{N}$ of $x$ on which $f_i\equiv\infty$ holds for any $i$ such that $x\notin\dom f_i$.
There is $\varsigma'>0$ such that $x+\varsigma d\in\mathcal{N}$ and $f_i(x+\varsigma d)<\infty$ for any $0<\varsigma<\varsigma'$ and $i\in I(x,d)$.
Thus, we have
\begin{equation}
F(x+\varsigma d)=\min_{i\in I(x,d)}f_i(x+\varsigma d)
\end{equation}
for all $0<\varsigma<\varsigma'$.
Note that $I(x,d)$ is not empty because $d\in\mathcal{F}(x;\dom F)$.
Using the continuity of $F(x+\varsigma d)$ and $f_i(x+\varsigma d)$ for $i\in I(x,d)$ at $\varsigma=0$, we obtain
\begin{equation}
F(x)=\lim_{\varsigma\searrow0}F(x+\varsigma d)=\lim_{\varsigma\searrow0}\min_{i\in I(x,d)}f_i(x+\varsigma d)=\min_{i\in I(x,d)}\lim_{\varsigma\searrow0}f_i(x+\varsigma d)=\min_{i\in I(x,d)}f_i(x).
\end{equation}
Again from the continuity of $f_i(x+\varsigma d)$ for $i\in I(x,d)$, there exists $\varsigma''>0$ such that
\begin{equation}
F(x+\varsigma d)=\min_{i\in I^*(x,d)}f_i(x+\varsigma d)
\end{equation}
for all $0\le\varsigma<\varsigma''$.
Consequently, it holds that
\begin{align}
F'(x;d)=\lim_{\varsigma\searrow0}\frac{F(x+\varsigma d)-F(x)}{\varsigma} &=\lim_{\varsigma\searrow0}\frac{\min_{i\in I^*(x,d)}\{f_i(x+\varsigma d)-f_i(x)\}}{\varsigma}\\
&=\min_{i\in I^*(x,d)}\lim_{\varsigma\searrow0}\frac{f_i(x+\varsigma d)-f_i(x)}{\varsigma}\\
&=\min_{i\in I^*(x,d)}f_i'(x;d),
\end{align}
which completes the proof.
\end{proof}

Using the above and the following lemmas (the latter one is a slight modification of \citet{delfour2019introduction}), we will show a property for a class of functions.

\begin{lemma}[cf. proof of $\mbox{\citep[Theorem 4.5, p.130]{delfour2019introduction}}$]\label{lem:d-derivative-descent-lemma}
Let $f:\mathbb{E}\to(-\infty,\infty]$ be convex and directionally differentiable.
The inequality
\begin{equation}
f(y)\ge f(x)+f'(x;y-x)
\end{equation}
holds for all $x,y\in\dom f$.
\end{lemma}

\begin{proposition}\label{prop:PMFMcvx=>LFcvx}
Let $f_i:\mathbb{E}\to(-\infty,\infty]$ be continuous on its domain, convex, and directionally differentiable and $\dom f_i$ be closed for $i=1,\ldots,M$.
Suppose that $F(x)\coloneqq\min_{i\in[M]}f_i(x)$ is continuous on $\dom F$.
Then, for any $x\in\dom F$, there exists a neighborhood $\mathcal{N}$ of $x$ such that $y-x\in\mathcal{F}(x;\dom F)$ and
\begin{equation}
F(y)\ge F(x)+F'(x;y-x)
\end{equation}
for all $y\in\mathcal{N}\cap\dom F$.
\end{proposition}

\begin{proof}
Let $I^*(x)=\{i\in[M]\mid f_i(x)=F(x)\}$ and $\overline{F}$ be a real number satisfying $F(x)<\overline{F}<\inf\{f_i(x)\mid f_i(x)>F(x)\}$.
From the continuity of $F$ and $f_i$, and the closedness of $\dom f_i$, there is a neighborhood $\mathcal{N}$ of $x$ such that $F(y)<\overline{F}<f_i(y)$ holds for any $y\in\mathcal{N}\cap\dom F$ and $i\notin I^*(x)$.
This implies
\begin{equation}
F(y)=\min_{\substack{i\in I^*(x)\\f_i(y)<\infty}}f_i(y)
\end{equation}
for $y\in\mathcal{N}\cap\dom F$.
For all $i\in I^*(x)$ such that $f_i(y)<\infty$, we see from convexity of $\dom f_i$ and $x\in\dom f_i$ that $y-x\in\mathcal{F}(x;\dom f_i)\subset\mathcal{F}(x;\dom F)$.
Consequently, we have
\begin{align}
F(y)-F(x) &=\min_{\substack{i\in I^*(x,y-x)\\f_i(y)<\infty}}\{f_i(y)-f_i(x)\}\\
&\ge\min_{\substack{i\in I^*(x,y-x)\\f_i(y)<\infty}}f_i'(x;y-x)\\
&\ge\min_{i\in I^*(x,y-x)}f_i'(x;y-x)\\
&=F'(x;y-x)
\end{align}
for $y\in\mathcal{N}\cap\dom F$, where the first inequality follows from Lemma \ref{lem:d-derivative-descent-lemma} and the last equality follows from Lemma \ref{lem:d-derivative-min}.
\end{proof}

The above proposition implies that the class of functions defined by the pointwise minimum of finitely many convex functions is a subclass of locally first-order convex functions (see Section 4 of \citet{chang2017local}).
The following is an immediate corollary of Proposition \ref{prop:PMFMcvx=>LFcvx}.

\begin{corollary}\label{cor:local-opt<=>d-stationary}
Suppose that all assumptions of Proposition \ref{prop:PMFMcvx=>LFcvx} hold.
If $x^*\in\dom F$ is a d-stationary point of \eqref{problem:general}, then $x^*$ is locally optimal to \eqref{problem:general}.
\end{corollary}

Corollary \ref{cor:local-opt<=>d-stationary} ensures that under the assumption any d-stationary point is also locally optimal without information about second order derivatives.
Since a convex function is continuous and directionally differentiable at the interior points of its domain, it is sufficient to assume only convexity of $f_i$ so as to ensure the equivalence especially when $\dom f_i=\mathbb{E}$ for all $i$.
The minimization problem we will deal with in the next section often has this property.

Lastly, let us mention the relation to existing results. 
\citet[Proposition 2.1]{ahn2017difference} showed that a d-stationary point of \eqref{problem:general} must be a local minimizer if $F$ is given by
\begin{equation}
F(x)=G(x)-H(x)+\delta_\mathcal{C}(x),
\end{equation}
where $G:\mathbb{R}^n\to\mathbb{R}$ is convex, $H:\mathbb{R}^n\to\mathbb{R}$ is piecewise affine convex, and $\mathcal{C}\subset\mathbb{R}^n$ is closed convex.
By Theorem 2.49 of \citet{rockafellar2009variational}, there exist finitely many affine functions $\{h_i\}_{i\in[M]}$ such that $H(x)=\max_{i\in[M]}h_i(x)$.
Consequently, it holds that
\begin{equation}
F(x)=G(x)-\max_{i\in[M]}h_i(x)+\delta_\mathcal{C}(x)=\min_{i\in[M]}\left\{G(x)-h_i(x)+\delta_\mathcal{C}(x)\right\}.
\end{equation}
Since $G(x)-h_i(x)+\delta_\mathcal{C}(x)$ is continuous on its domain $\mathcal{C}$, is convex, and is directionally differentiable, we see that Corollary \ref{cor:local-opt<=>d-stationary} is a generalization of Proposition 2.1 of \citet{ahn2017difference}.
Besides, \citet[Proposition 4.1]{cui2020study} showed that a d-stationary point of \eqref{problem:general} must be a local minimizer if $F=G\circ H+\delta_\mathcal{C}$ where $G:\mathbb{R}^m\to\mathbb{R}$ is convex, $H:\mathbb{R}^n\to\mathbb{R}^m$ is piecewise affine, and $\mathcal{C}\subset\mathbb{R}^n$ is closed convex.
Note that there are a finite number of polyhedra $\{\mathcal{D}_i\}_{i\in[M]}$ and corresponding affine functions $\{h_i\}_{i\in[M]}$ such that $\mathbb{R}^n=\cup_{i\in[M]}\mathcal{D}_i$ and $H(x)=h_i(x)$ on $\mathcal{D}_i$ \citep[Proposition 4.2.1]{facchinei2003finite}.
Let $f_i(x)=G\circ h_i(x)+\delta_{\mathcal{D}_i\cap\mathcal{C}}(x)$, then it holds that
\begin{equation}
F(x)=\min_{i\in[M]}f_i(x),
\end{equation}
which implies that Proposition 4.1 of \citet{cui2020study} is a corollary of Corollary \ref{cor:local-opt<=>d-stationary}.

\section{Unconstrained Problem Penalized by Trimmed $\ell_1$ Norm}\label{sec:GTL}
In this section, we introduce a unified framework of trimmed $\ell_1$ penalized problem and show the equivalence with the corresponding structured sparsity constrained problem.

\subsection{Trimmed $\ell_1$ Norm}
Before describing the formulations, let us start with the key component that plays a significant role in those formulations. 
For $z=(z_1^\top,\ldots,z_m^\top)^\top\in\mathbb{R}^{mp}$ with $z_1,\ldots,z_m\in\mathbb{R}^p$, the trimmed group $\ell_1$ norm \citep{yagishita2024pursuit} is defined by
\begin{equation}\label{eq:trimmed-l1-norm}
T_{K,m,p}(z)\coloneqq\min_{\substack{\Lambda\subset[m]\\|\Lambda|=m-K}}\sum_{i\in\Lambda}\|z_i\|_2,
\end{equation}
where $K\in\{0,\ldots,m-1\}$.
It is not hard to see that \eqref{eq:trimmed-l1-norm} is equal to the sum of the smallest $m-K$ components of the $m$-dimensional vector $(\|z_1\|_2,\ldots,\|z_m\|_2)$.
Accordingly, the trimmed group $\ell_1$ norm with $K=0$ is the group $\ell_1$ norm, namely, $T_{0,m,p}(z)=\sum_{i=1}^{m}\|z_i\|_2$.
Note also that when $p=1$, the function $T_{K,m,p}$ is known as the trimmed $\ell_1$ norm, namely, the sum of the smallest $m-K$ components of a vector $(|z_1|,\ldots,|z_m|)$ \citep{luo2013new,huang2015two,gotoh2018dc}.
In the following, ``group'' in the term ``trimmed group $\ell_1$ norm'' is omitted for simplicity.

It is easy to see that $T_{K,m,p}(z)\ge0$ for any $z$, and that $T_{K,m,p}(z)=0$ if and only if
\begin{equation}\label{ineq:simple-card}
\left\|(\|z_1\|_2,\ldots,\|z_m\|_2)\right\|_0\le K.
\end{equation}
These facts motivate us to use the trimmed $\ell_1$ norm as a penalty function of the cardinality constraint \eqref{ineq:simple-card}.
The idea of using the trimmed $\ell_1$ penalty was first utilized by 
\citet{gotoh2018dc}, whereas the equivalence with the cardinality constraint was pointed out earlier in the proof of a proposition of \citet{hempel2014novel}.
In the following, we will use \eqref{eq:trimmed-l1-norm} as the penalty function as a surrogate of the corresponding cardinality constraints. 

Noting that the function $z\mapsto \sum_{i\in\Lambda}\|z_i\|_2$ is directionally differentiable on $\mathbb{R}^{mp}$ and Lemma \ref{lem:d-derivative-min}, we readily see that $T_{K,m,p}$ is directionally differentiable on $\mathbb{R}^{mp}$. 
A more specific expression than the one in Lemma \ref{lem:d-derivative-min} of the directional derivative of $T_{K,m,p}$ is given as follows.

\begin{lemma}[$\mbox{\citet[Lemma 11]{yagishita2024pursuit}}$]\label{lem:d-derivative-T_K}
Let $\Lambda_1=\{i\mid\|z_i\|_2<\|z_{(K)}\|_2\}$ and $\Lambda_2=\{i\mid\|z_i\|_2=\|z_{(K)}\|_2\}$. The directional derivative of $T_K$ at $z\in\mathbb{R}^{mp}$ in the direction $d\in\mathbb{R}^{mp}$ is given by
\begin{equation}\label{eq:d-diff}
T_{K,m,p}'(z;d)=\sum_{i\in\Lambda_1}\Delta(z_i;d_i)+\min_{\substack{\Lambda\subset\Lambda_2 \\ |\Lambda|=m-K-|\Lambda_1|}}\sum_{i\in\Lambda}\Delta(z_i;d_i),
\end{equation}
where
\begin{align}
\Delta(z_i;d_i)\coloneqq
\begin{cases}
 \frac{z_i^\top d_i}{\|z_i\|_2}, & z_i\neq0,\\
 \|d_i\|_2, & z_i=0
\end{cases}
\end{align}
and $\|z_{(K)}\|_2$ is the $K$th largest component of $(\|z_1\|_2,\ldots,\|z_m\|_2)$ if $K\in[m-1]$, $\|z_{(0)}\|_2=\infty$ if $K=0$.
\end{lemma}
This representation for $p=1$ was given by \citet[Lemma SM2.6]{amir2021trimmed} and later extended to $p>1$ by \citet{yagishita2024pursuit}.
Lemma \ref{lem:d-derivative-T_K} will be used later in this section to show the equivalence between our proposed formulation and the corresponding structured sparsity constrained problem, both of which will be given below.

\subsection{Formulation of Generalized Trimmed Lasso}
We are now ready to present the formulation to be examined in this paper. 
With the trimmed $\ell_1$ norm, we consider the following unconstrained minimization problem, which we refer to as the {\it generalized trimmed lasso}:
\begin{equation}\label{problem:GTL}
\underset{x_0\in\mathbb{R}^{n_0},x_1\in\mathbb{R}^{n_1},\ldots,x_L\in\mathbb{R}^{n_L}}{\mbox{minimize}} \quad f(x_0,x_1,\ldots,x_L)+\sum_{l\in[L]} \gamma_lT_{K_l,m_l,p_l}(D_lx_l-c_l),
\end{equation}
where $\gamma_l>0,~ K_l\in\{0,\ldots,m_l-1\},~ c_l\in\mathbb{R}^{m_lp_l}$, and $D_l\neq 0$ is an $m_lp_l\times n_l$ matrix.
For simplicity, we ignore the subscripts when $L=1$ and $n_0=0$.
The first term of the objective function of \eqref{problem:GTL} is usually a loss function for fitting models to some data, and the second term is the sum of penalty functions of the corresponding structured sparsity constraints. 
The coefficients of the trimmed $\ell_1$ penalty terms, $\gamma_l$'s, look like hyperparameters to take the balance among the first term and the (multiple) trimmed $\ell_1$ penalties.
As will be discussed later in detail, however, those can be treated as exact penalties under some conditions and may not actually need to be calibrated.

Although the generalized trimmed lasso \eqref{problem:GTL} includes the nonconvex nonsmooth functions, $T_{K_l,m_l,p_l}$, in its objective, we can see that the equivalence between the d-stationarity and the local optimality holds if the first term is convex.    

\begin{proposition}\label{prop:local-opt<=>d-stationary-in-GTL}
d-stationary points and local minimizers of \eqref{problem:GTL} are equivalent if $f$ is convex.
\end{proposition}

This proposition is readily from the definition of $T_{K,m,p}$ and Corollary \ref{cor:local-opt<=>d-stationary}.
The generalized trimmed lasso \eqref{problem:GTL} covers a lot of existing trimmed $\ell_1$ penalized problems \citep{luo2013new,huang2015two,gotoh2018dc,ahn2017difference,banjac2017novel,bertsimas2017trimmed,lu2018sparse,yun2019trimming,nakayama2021superiority,amir2021trimmed,yagishita2024pursuit,lu2023exact}, as follows.
Note that all the following examples satisfy the assumption of Proposition \ref{prop:local-opt<=>d-stationary-in-GTL}, namely, the convexity of $f$.

\begin{example}[Sparse linear regression \citep{luo2013new,huang2015two,gotoh2018dc,banjac2017novel,bertsimas2017trimmed,yun2019trimming,amir2021trimmed}]\label{ex:sparse-ols}
Let $b\in\mathbb{R}^q,~ A\in\mathbb{R}^{q\times n}$ denote the output vector and the design matrix for the ordinary least square estimation. For $\eta\ge0$, \citet{huang2015two} considered the following problem:
\begin{equation}\label{problem:sparse-ols}
\underset{x\in\mathbb{R}^n}{\mbox{minimize}} \quad \frac{1}{2}\|b-Ax\|_2^2+\eta\|x\|_1+\gamma T_{K,n,1}(x).
\end{equation}
The first two terms of \eqref{problem:sparse-ols} correspond to $f$ of \eqref{problem:GTL}. 
\citet{gotoh2018dc,banjac2017novel,yun2019trimming}, and \citet{amir2021trimmed} dealt with the problem \eqref{problem:sparse-ols} with $\eta=0$, while \citet{luo2013new} and \citet{bertsimas2017trimmed} dealt with the problem \eqref{problem:sparse-ols} with $\eta>0$.
To the best of our knowledge, the first penalized problems via the trimmed $\ell_1$ norm were proposed in this form independently by \citet{luo2013new,huang2015two}, and \citet{gotoh2018dc}.
\end{example}

As mentioned in the previous subsection, since the trimmed $\ell_1$ norm is reduced to the $\ell_1$ norm if $K=0$, the generalized trimmed lasso \eqref{problem:GTL} is also an extension of the generalized lasso \citep{tibshirani2011solution,she2010sparse}.
Hence, in addition to Example \ref{ex:sparse-ols}, the generalized trimmed lasso \eqref{problem:GTL} has many applications of trimmed $\ell_1$ penalty that have not yet been addressed.

\begin{example}[Sparse linear regression with power $1$ \citep{amir2021trimmed}]\label{ex:sparse-reg-p1}
For $b\in\mathbb{R}^q$ and $A\in\mathbb{R}^{q\times n}$ given in the same manner as in Example \ref{ex:sparse-ols}, \citet{amir2021trimmed} considered the following sparse regression problem with power $1$:
\begin{equation}\label{problem:sparse-reg-p1}
\underset{x\in\mathbb{R}^n}{\mbox{minimize}} \quad \|b-Ax\|_2+\gamma T_{K,n,1}(x).
\end{equation}
While only difference is in the first term of the objective from Example \ref{ex:sparse-ols}, this case satisfied the Lipschitz continuity.
\end{example}

\begin{example}[Sparse logistic regression \citep{gotoh2018dc,lu2018sparse,nakayama2021superiority}]\label{ex:sparse-logistic}
Let $b_j\in\{-1,1\},~ a_j\in\mathbb{R}^{n},~ j\in[q]$ denote, respectively, the binary label and attribute vector of the $j$th sample. To conduct the logistic regression with variable selection, \citet{gotoh2018dc} considered the following problem:
\begin{equation}\label{problem:sparse-logistic}
\underset{x\in\mathbb{R}^n}{\mbox{minimize}} \quad \sum_{j\in[q]}\log\left(1+\exp(-b_ja_j^\top x)\right)+\gamma T_{K,n,1}(x).
\end{equation}
The difference from Examples \ref{ex:sparse-ols} and \ref{ex:sparse-reg-p1} is only at the first term, called {\it logistic loss}, and this case has a Lipschitz gradient and Lipschitz continuity. 
\citet{lu2018sparse} and \citet{nakayama2021superiority} also treated this problem.
\end{example}

\begin{example}[Sparse SVM \citep{gotoh2018dc}]\label{ex:sparse-SVM}
For $b_j\in\{-1,1\},~ a_j\in\mathbb{R}^{n},~ j\in[q]$ given in the same manner as in Example \ref{ex:sparse-logistic}, \citet{gotoh2018dc} also considered the following problem:
\begin{equation}\label{problem:sparse-SVM}
\underset{x\in\mathbb{R}^n}{\mbox{minimize}} \quad \sum_{j\in[q]}\max\{1-b_ja_j^\top x,0\}+\gamma T_{K,n,1}(x).
\end{equation}
The first term, called hinge loss, is Lipschitz continuous but is not differentiable.
\end{example}

\begin{example}[Sparse robust regression \citep{nakayama2021superiority}]\label{ex:sparse-robust-reg}
Let $b\in\mathbb{R}^{n_2}$ and $A\in\mathbb{R}^{n_2\times n_1}$ be given in the same manner as in Examples \ref{ex:sparse-ols} and \ref{ex:sparse-reg-p1}, \citet{nakayama2021superiority} considered the following problem:
\begin{equation}\label{problem:sparse-robust-reg}
\underset{x_1\in\mathbb{R}^{n_1}, x_2\in\mathbb{R}^{n_2}}{\mbox{minimize}} \quad \frac{1}{2}\|b-Ax_1-x_2\|_2^2+\gamma_1 T_{K_1,n,1}(x_1)+\gamma_2 T_{K_2,q,1}(x_2).
\end{equation}
While the first trimmed $\ell_1$ penalty term is for variable selection, the second trimmed $\ell_1$ penalty is for removing the outlying samples. 
Here, the variable $x_2$ serves to absorb residuals of outliers.
\end{example}

\begin{example}[Minimizing constraints violation]\label{ex:MCV}
Let $b\in\mathbb{R}^n,~ c\in\mathbb{R}^m$ and $D\in\mathbb{R}^{m\times n}$ be a surjective matrix. We consider the following problem:
\begin{equation}\label{problem:MCV}
\underset{x\in\mathbb{R}^n}{\mbox{minimize}} \quad \frac{1}{2}\|b-x\|_2^2+\gamma T_{K,m,1}(Dx-c).
\end{equation}
\citet{li2015global} dealt with the corresponding cardinality constrained problem.
\end{example}

\begin{example}[Estimation of piece-wise linear trend]\label{ex:trend-filtering}
Let $b\in\mathbb{R}^n$ be a vector of time series data samples. 
To estimate a piecewise linear function to approximate the data points $b_1,...,b_n$, 
consider the following problem:
\begin{equation}\label{problem:trend-filtering}
\underset{x\in\mathbb{R}^n}{\mbox{minimize}} \quad \frac{1}{2}\|b-x\|_2^2+\gamma_1 T_{K,n-2,1}(D^{(2,n)}x),
\end{equation}
where $D^{(2,n)}\in\mathbb{R}^{(n-2)\times n}$ denotes the second-order difference matrix defined by
\begin{equation}
D^{(2,n)}\coloneqq
\begin{pmatrix}
1 & -2 & 1 & & & \\
& 1 & -2 & 1 & & \\
& & \ddots & \ddots & \ddots & \\
& & & 1 & -2 & 1 \\
\end{pmatrix}.
\end{equation}
The formulation \eqref{problem:trend-filtering} is an extension of $\ell_1$ trend filtering \citep{kim2009ell_1} by the trimmed $\ell_1$ norm.
\end{example}

The above examples are for $p=1$ only, but it is natural to extend to $p>1$ in each.
The network trimmed lasso \citep{yagishita2024pursuit} is such a kind and it 
pursues both clustering data samples and fitting the models to the samples simultaneously.
See \citet{yagishita2024pursuit} for more details.

It is a natural question to ask how the relationship between solutions of the generalized trimmed lasso and the following structured sparsity constrained problem:
\begin{align}
\underset{x_0\in\mathbb{R}^{n_0},x_1\in\mathbb{R}^{n_1},\ldots,x_L\in\mathbb{R}^{n_L}}{\mbox{minimize}} & \qquad f(x_0,x_1,\ldots,x_L) \label{obj:card}\\
\text{subject to} ~\quad & \qquad T_{K_l,m_l,p_l}(D_lx_l-c_l)=0, \quad l=1,\ldots,L. \label{const:card}
\end{align}

Due to the property explained above \eqref{ineq:simple-card}, any optimal (resp. locally optimal, d-stationary) solution of the generalized trimmed lasso \eqref{problem:GTL} satisfying the constraint \eqref{const:card} is optimal (resp. locally optimal, d-stationary) to the problem \eqref{obj:card}--\eqref{const:card}. 
On the other hand, as $\gamma_l$ is larger (while leaving the other components the same), the term $T_{K_l,m_l,p_l}(D_lx_l^*-c_l)$ is smaller for the obtained solution $(x_0^*,x_1^*,\ldots,x_L^*)$ of \eqref{problem:GTL}, and one might expect $T_{K_l,m_l,p_l}(D_lx_l^*-c_l)=0$ holds for a sufficiently large $\gamma_l$. 
In fact, for special cases of generalized trimmed lasso, several studies \citep{gotoh2018dc,ahn2017difference,bertsimas2017trimmed,nakayama2021superiority,amir2021trimmed,yagishita2024pursuit,lu2023exact} have derived thresholds $\overline{\gamma}_l$ such that if $\gamma_l>\overline{\gamma}_l$, then any globally optimal (or locally optimal) solution $(x_0^*,x_1^*,\ldots,x_L^*)$ of \eqref{problem:GTL} satisfies $T_{K_l,m_l,p_l}(D_lx_l^*-c_l)=0$. 
Once such $\overline{\gamma}_l$'s are calculated for all $l\in[L]$, a structured sparse solution satisfying \eqref{const:card} can be obtained by solving the generalized trimmed lasso \eqref{problem:GTL} with $\gamma_l>\overline{\gamma}_l$ for $l\in[L]$. 
Such $\overline{\gamma}_l$'s are called {\it exact penalty parameters}.

Assumptions of the existing results are briefly summarized in Table \ref{table:previous-result}\footnote{In the table, $b\in\mathbb{R}^q,~ A\in\mathbb{R}^{q\times n},~ \eta>0$, and ``boundedness'' means boundedness of the set of solutions.
Since $c=0$ is supposed in all the previous works, it is omitted in the table.
The results of \citet{amir2021trimmed} and \citet{yagishita2024pursuit} were valid at local minimizers, but since both of them assumed the convexity of $f$, we see from Proposition \ref{prop:local-opt<=>d-stationary-in-GTL} that it is essentially the same as the result at d-stationary points.
In \citet{yagishita2024pursuit}, the result was also given at globally optimal solutions without assuming the convexity of $f$.}.
The weakness of the previous results can be concisely summarized as follows:
\begin{itemize}
\item Since the trimmed $\ell_1$ norm is a nonconvex function (unless $K=0$), it is generally difficult to ensure the global optimality in practice, and the implication of obtained solutions for the corresponding cardinality constrained problems would be unclear even when the exact penalty parameters at globally optimal solutions such as in the result \cite{gotoh2018dc,bertsimas2017trimmed,nakayama2021superiority,lu2023exact} are employed. 
\item At first glance, the propositions in the papers \citep{gotoh2018dc,ahn2017difference,nakayama2021superiority,yagishita2024pursuit} seem to cover a number of cases, but they all assume the ``boundedness'' of the set of solutions and it is not necessarily fulfilled or difficult to ensure the fulfillment in many cases. 
\item The convexity of $f$ is assumed in all the existing studies for establishing the exact penalty result at locally optimal solutions.
\end{itemize}

To mitigate the drawback in the first bullet, we will show the exact penalties at $d$-stationary points. 
As for the second and third bullets, we will remove or relax the assumptions. 
In addition, since it seems difficult to apply the proof approaches used to derive the existing results to the generalized trimmed lasso \eqref{problem:GTL}, we will build a novel unified approach for the generalized framework.

\begin{table}
  \caption{Assumptions of previous works.}\small
  \label{table:previous-result}
  \centering
  \begin{tabular}{|l||c|c|c|c|c|} \hline
   & solution & $L$ & $f$ & $D$ & other \\ \hline\hline
   \citet{gotoh2018dc} & optimal & 1 & $M$-smooth & identity matrix & boundedness \\ \hline
   \citet{ahn2017difference} & d-stationary & 1 & $M$-smooth, convex & identity matrix & boundedness \\ \hline
   \citet{bertsimas2017trimmed} & optimal & 1 & $\frac{1}{2}\|b-Ax\|_2^2+\eta\|x\|_1$ & identity matrix & none \\ \hline
   \citet{nakayama2021superiority} & optimal & 2 & $M$-smooth & identity matrices & boundedness \\ \hline
   \citet{amir2021trimmed} & d-stationary & 1 & $\frac{1}{2}\|b-Ax\|_2^2,~ \|b-Ax\|_2$ & identity matrix & none \\ \hline
   \citet{yagishita2024pursuit} & d-stationary & 1 & $M$-smooth, convex & specific non-identity & boundedness \\ \hline
   \citet{lu2023exact} & optimal & 1 & Lipschitz continuous & identity matrix & none \\ \hline
  \end{tabular}
\end{table}

\subsection{Exact Penalties for Generalized Trimmed Lasso}
In the remainder of this section, we derive exact penalty parameters at d-stationary points of the generalized trimmed lasso \eqref{problem:GTL} under milder conditions than the existing results.
Since our proof does not rely on the convexity of $f$, we see that the assumption of convexity is not essential in deriving exact penalty parameters at practically attainable solutions (d-stationary points in our analysis) of the generalized lasso \eqref{problem:GTL}.
In addition, we will see later that our results are stronger than some existing ones, even though the proof is for a more general formulation. 

To derive exact penalty parameters for the generalized trimmed lasso \eqref{problem:GTL}, we present the following lemma.
For $\Lambda\subset[m]$ and $D=\left((D)_1^\top,\ldots,(D)_m^\top\right)^\top\neq 0$ with $(D)_i\in\mathbb{R}^{p\times n}$, we denote the submatrix leaving only $(D)_i$ corresponding to $\Lambda$ by $(D)_\Lambda$, and the positive constant depending on $D$ is defined as follows:
\begin{align}
\sigma_{K,m,p}(D)\coloneqq
\begin{cases}
\qquad\qquad\qquad\qquad\qquad\sigma_{\min}(D), & \mbox{if $D$ is surjective},\\
\displaystyle\min\{\sigma_{\min}((D)_\Lambda)\mid\Lambda\subset[m], |\Lambda|=m-K, (D)_\Lambda\neq 0\}, & \mbox{if $D$ is not surjective}.
\end{cases}
\end{align}
We remark that $\sigma_{K,m,p}(I)=\sigma_{\min}(I)=1$.

\begin{lemma}\label{lem:epp-GTL}
Let $f$ be directionally differentiable and $(x_0^*,x_1^*,\ldots,x_L^*)$ be a d-stationary point of \eqref{problem:GTL}.
Suppose that the following assumptions hold for $l\in[L]$:
\begin{enumerate}[({A}1)]
\item There is an $\overline{x}_l\in\mathbb{R}^{n_l}$ such that $D_l\overline{x}_l-c_l=0$;
\item There exists $\Gamma_l>0$ such that
\begin{align}\label{ineq:bounded-cond-GTL}
\sup_{\|d_l\|_2=1}f'(x_0^*,\ldots,x_L^*;0,\ldots,d_l,\ldots,0)\le\Gamma_l.
\end{align}
\end{enumerate}
Then, $(x_0^*,x_1^*,\ldots,x_L^*)$ satisfies $T_{K_l,m_l,p_l}(D_lx_l^*-c_l)=0$ if
\begin{align}\label{ineq:epp-GTL}
\gamma_l>\frac{1}{\sigma_{K_l,m_l,p_l}(D_l)}\Gamma_l
\end{align}
holds.
When $p_l=1,~ D_l=I,~ c_l=0$, and $f(x_0,x_1,\ldots,x_L)=g(x_0,x_1,\ldots,x_L)+\sum_{l\in[L]}\eta_l\|x_l\|_1$ with $\eta_l\ge0$ hold, the left hand side of \eqref{ineq:bounded-cond-GTL} can be replaced by
\begin{align}
\sup_{\substack{\|d_l\|_1=1\\ d_l\in\{-1,0,1\}^{n_l}}}g'(x_0^*,\ldots,x_L^*;0,\ldots,d_l,\ldots,0)-\eta_l.
\end{align}
\end{lemma}

\begin{proof}
To derive a contradiction, we assume that $T_{K_l,m_l,p_l}(D_lx_l^*-c_l)>0$.
It holds that
\begin{equation}
f'(x_0^*,\ldots,x_L^*;0,\ldots,d_l,\ldots,0)+\gamma_lT_{K_l,m_l,p_l}'(D_lx_l^*-c_l;D_ld_l)\ge0
\end{equation}
for any $d_l\in\mathbb{R}^{n_l}$ because of the d-stationarity of $(x_1^*,\ldots,x_L^*)$.
By Lemma \ref{lem:d-derivative-T_K}, we see that
\begin{equation}\label{ineq:f-o-condition}
f'(x_0^*,\ldots,x_L^*;0,\ldots,d_l,\ldots,0)+\gamma_l\sum_{i\in\Lambda_1\cup\Lambda}\Delta((D_lx_l^*-c_l)_i;(D_ld_l)_i)\ge0
\end{equation}
for some $\Lambda\subset\Lambda_2$ such that $|\Lambda|=m_l-K_l-|\Lambda_1|$ and any $d_l\in\mathbb{R}^{n_l}$.
Note that $(D_lx_l^*-c_l)_i\neq0$ holds for some $i\in\Lambda_1\cup\Lambda$, since $T_{K_l,m_l,p_l}(D_lx_l^*-c_l)>0$ is assumed.
We will divide the case into (i) $D_l$ is surjective, and (ii) $D_l$ is not surjective.

\noindent Case (i): 
Since $D_l$ is surjective, there exists $d\in\ker(D_l)^\perp$ such that
\begin{align}
(D_ld)_i=
\begin{cases}
 -(D_lx_l^*-c_l)_i, & i\in\Lambda_1\cup\Lambda,\\
 0, & i\notin\Lambda_1\cup\Lambda.
\end{cases}
\end{align}
Note that $d\neq0$, because $(D_lx_l^*-c_l)_i\neq0$ holds for some $i\in\Lambda_1\cup\Lambda$.
The inequality \eqref{ineq:f-o-condition} with $d_l=d$ yields
\begin{align}
\begin{split}
f'(x_0^*,\ldots,x_L^*;0,\ldots,d,\ldots,0) &\ge\gamma_l\sum_{i\in\Lambda_1\cup\Lambda}\|(D_ld)_i\|_2\\
&=\gamma_l\sum_{i\in[m_l]}\|(D_ld)_i\|_2\\
&\ge\gamma_l\left(\sum_{i\in[m_l]}\|(D_ld)_i\|_2^2\right)^\frac{1}{2}\\
&=\gamma_l\|D_ld\|_2\\
&\ge\gamma_l\sigma_{\min}(D_l)\|d\|_2\\
&=\gamma_l\sigma_{K_l,m_l,p_l}(D_l)\|d\|_2,
\end{split}
\end{align}
where the third inequality follows from Lemma \ref{lem:restricted bounded below} and the third equality from the definition of $\sigma_{K_l,m_l,p_l}$.
Then, from the positive homogeneity of the directional derivative and the Assumption (A2), we have
\begin{equation}
\gamma_l\le\frac{1}{\sigma_{K_l,m_l,p_l}(D_l)}f'\left(x_0^*,\ldots,x_L^*;0,\ldots,\frac{d}{\|d\|_2},\ldots,0\right)\le\frac{1}{\sigma_{K_l,m_l,p_l}(D_l)}\Gamma_l,
\end{equation}
which contradicts the inequality \eqref{ineq:epp-GTL}.

\noindent Case (ii):
By the Assumption (A1), $D_l(x_l^*-\overline{x}_l)=D_lx_l^*-c_l$ holds.
Let $d\in\ker((D_l)_{\Lambda_1\cup\Lambda})^\perp$ be the projection of $-(x_l^*-\overline{x}_l)$ onto $\ker((D_l)_{\Lambda_1\cup\Lambda})^\perp$, then we have
\begin{align}
(D_ld)_{\Lambda_1\cup\Lambda}=(D_l)_{\Lambda_1\cup\Lambda}d=-(D_l)_{\Lambda_1\cup\Lambda}(x_l^*-\overline{x}_l)=(-D_l(x_l^*-\overline{x}_l))_{\Lambda_1\cup\Lambda}=(-(D_lx_l^*-c_l))_{\Lambda_1\cup\Lambda}.
\end{align}
Note that $(D_l)_{\Lambda_1\cup\Lambda}\neq 0$ and $d\neq0$, because $(D_lx_l^*-c_l)_i\neq0$ holds for some $i\in\Lambda_1\cup\Lambda$.
The inequality \eqref{ineq:f-o-condition} with $d_l=d$ yields
\begin{align}
\begin{split}
f'(x_0^*,\ldots,x_L^*;0,\ldots,d,\ldots,0) &\ge\gamma_l\sum_{i\in\Lambda_1\cup\Lambda}\|(D_ld)_i\|_2\\
&\ge\gamma_l\left(\sum_{i\in\Lambda_1\cup\Lambda}\|(D_ld)_i\|_2^2\right)^\frac{1}{2}\\
&=\gamma_l\|(D_ld)_{\Lambda_1\cup\Lambda}\|_2\\
&=\gamma_l\|(D_l)_{\Lambda_1\cup\Lambda}d\|_2\\
&\ge\gamma_l\sigma_{\min}((D_l)_{\Lambda_1\cup\Lambda})\|d\|_2\\
&\ge\gamma_l\sigma_{K_l,m_l,p_l}(D_l)\|d\|_2,
\end{split}
\end{align}
where the third inequality follows from Lemma \ref{lem:restricted bounded below} and the fourth one from the definition of $\sigma_{K_l,m_l,p_l}$.
The contradiction to \eqref{ineq:epp-GTL} is derived just as when $D_l$ is surjective.

Lastly, we consider the case where $p_l=1,~ D_l=I,~ c_l=0$, and $f(x_0,x_1,\ldots,x_L)=g(x_0,x_1,\ldots,x_L)+\sum_{l\in[L]}\eta_l\|x_l\|_1$ with $\eta_l\ge0$ hold.
We assume that $T_{K_l,n_l,1}(x_l^*)>0$.
By Lemma \ref{lem:d-derivative-T_K} and the d-stationarity of $(x_1^*,\ldots,x_L^*)$, it holds that
\begin{equation}\label{ineq:f-o-condition-2}
g'(x_0^*,\ldots,x_L^*;0,\ldots,d_l,\ldots,0)+\eta_l\sum_{i\in[n_l]}\Delta((x_l^*)_i;(d_l)_i)+\gamma_l\sum_{i\in\Lambda_1\cup\Lambda}\Delta((x_l^*)_i;(d_l)_i)\ge0
\end{equation}
for some $\Lambda\subset\Lambda_2$ such that $|\Lambda|=n_l-K_l-|\Lambda_1|$ and any $d_l\in\mathbb{R}^{n_l}$.
Note that there exists $i'\in\Lambda_1\cup\Lambda$ such that $(x_l^*)_{i'}\neq0$ because $T_{K_l,n_l,1}(x_l^*)>0$ is assumed.
Substituting $d=-\frac{(x_l^*)_{i'}}{|(x_l^*)_{i'}|}e_{i'}$ back into \eqref{ineq:f-o-condition-2} leads to
\begin{equation}
g'\left(x_0^*,\ldots,x_L^*;0,\ldots,-\frac{(x_l^*)_{i'}}{|(x_l^*)_{i'}|}e_{i'},\ldots,0\right)+(\eta_l+\gamma_l)\left(-\frac{(x_l^*)_{i'}(x_l^*)_{i'}}{|(x_l^*)_{i'}||(x_l^*)_{i'}|}\right)\ge0.
\end{equation}
By this inequality, we obtain that
\begin{equation}
\gamma_l\le\sup_{\substack{\|d_l\|_1=1\\ d_l\in\{-1,0,1\}^{n_l}}}g'(x_0^*,\ldots,x_L^*;0,\ldots,d_l,\ldots,0)-\eta_l\le\Gamma_l,
\end{equation}
which contradicts the inequality $\gamma_l>\Gamma_l$.
\end{proof}

Note that the Assumption (A1) is established whenever $c_l=0$ or $D_l$ is surjective.
On the other hand, considering the case where $L=1$, $n_0=0$ and $f$ is differentiable, we have
\begin{equation}
\sup_{\|d\|_2=1}f'(x^*;d)=\sup_{\|d\|_2=1}\nabla f(x^*)^\top d=\|\nabla f(x^*)\|_2.
\end{equation}
From this, we see that the Assumption (A2) can be interpreted as ``a boundedness assumption for the gradient of $f$ at stationary points,'' and is a natural one.
\citet{banjac2017novel} found a condition similar to the Assumption (A2) in a limited case, but they were not aware that the boundedness of the gradient of $f$ at local minimizers can guarantee the existence of the exact penalty parameter.
No propositions have revealed such an essential and intuitive assumption on $f$ in the previous papers \citep{gotoh2018dc,ahn2017difference,bertsimas2017trimmed,nakayama2021superiority,amir2021trimmed,yagishita2024pursuit,lu2023exact}.
We will now use Lemma \ref{lem:epp-GTL} to derive exact penalty parameters of the generalized trimmed lasso \eqref{problem:GTL}.

\begin{theorem}\label{thm:epp-smooth-bounded}
Suppose that the Assumption (A1) holds and $f$ is $M$-smooth. 
Let $(x_0^*,x_1^*,\ldots,x_L^*)$ be a d-stationary point of \eqref{problem:GTL} and $C_l>0$ be constants such that $\|x_l^*\|_2\le C_l$ for all $l=0,\ldots,L$.
Then, $(x_0^*,x_1^*,\ldots,x_L^*)$ satisfies the constraint \eqref{const:card} if
\begin{equation}\label{ineq:epp-smooth-bounded}
\gamma_l>\frac{1}{\sigma_{K_l,m_l,p_l}(D_l)}\left(\|\nabla_{x_l}f(0,\ldots,0)\|_2+M\sqrt{\sum_{l=0}^LC_l^2}\right)
\end{equation}
holds for all $l\in[L]$.
The right hand side of \eqref{ineq:epp-smooth-bounded} can be replaced by $\|\nabla_{x_l}f(0,\ldots,0)\|_\infty+M\sqrt{\sum_{l=0}^LC_l^2}$ when $p_l=1,~ D_l=I$, and $c_l=0$ hold.
\end{theorem}

\begin{proof}
By the $M$-smoothness of $f$, we have for all $d_l$ such that $\|d_l\|_2=1$ that 
\begin{align}
f'(x_0^*,\ldots,x_L^*;0,\ldots,d_l,\ldots,0) &=\nabla_{x_l}f(x_0^*,\ldots,x_L^*)^\top d_l\\
&\le\|\nabla_{x_l}f(x_0^*,\ldots,x_L^*)\|_2\\
&\le\|\nabla_{x_l}f(0,\ldots,0)\|_2+\|\nabla_{x_l}f(x_0^*,\ldots,x_L^*)-\nabla_{x_l}f(0,\ldots,0)\|_2\\
&\le\|\nabla_{x_l}f(0,\ldots,0)\|_2+M\|(x_0^*,\ldots,x_L^*)\|_2\\
&\le\|\nabla_{x_l}f(0,\ldots,0)\|_2+M\sqrt{\sum_{l=0}^LC_l^2},
\end{align}
where the first inequality follows from the Cauchy-Schwarz inequality, the second one from the triangle inequality, the third one from the $M$-smoothness of $f$, and the fourth one from the assumption $\|x_l^*\|_2\le C_l$.
Thus, the assumption \emph{(A2)} holds with $\Gamma_l=\|\nabla_{x_l}f(0,\ldots,0)\|_2+M\sqrt{\sum_{l=0}^LC_l^2}$.
From Lemma \ref{lem:epp-GTL}, we have the desired result.
When $p_l=1,~ D_l=I$, and $c_l=0$ hold, since we can evaluate as
\begin{align}
\nabla_{x_l}f(x_0^*,\ldots,x_L^*)^\top d_l &\le\|\nabla_{x_l}f(x_0^*,\ldots,x_L^*)\|_\infty\\
&\le\|\nabla_{x_l}f(0,\ldots,0)\|_\infty+\|\nabla_{x_l}f(x_0^*,\ldots,x_L^*)-\nabla_{x_l}f(0,\ldots,0)\|_2\\
&\le\|\nabla_{x_l}f(0,\ldots,0)\|_\infty+M\sqrt{\sum_{l=0}^LC_l^2}
\end{align}
for all $d_l$ such that $\|d_l\|_1=1$ and $d_l\in\{-1,0,1\}^{n_l}$, from Lemma \ref{lem:epp-GTL}, the right hand side of \eqref{ineq:epp-smooth-bounded} can be replaced.
\end{proof}

We compare Theorem \ref{thm:epp-smooth-bounded} for $L=1,~ n_0=0,~ p=1,~ D=I$, and $c=0$ with the results of \citet[Theorem 3]{gotoh2018dc} and \citet[Theorem 5.1]{ahn2017difference}\footnote{Looking at \citet{ahn2017difference}, it may appear at first glance that their result is a different claim, but with the appropriate parameter transformations, we can see that their result is the same one.}.
\begin{list}{}{}
\item[In \citep{gotoh2018dc}:] Assume that $f$ is $M$-smooth, $L=1,~ n_0=0,~ p=1,~ D=I$, and $c=0$.
Let $x^*$ be an optimal solution of \eqref{problem:GTL} and $C>0$ be a constant such that $\|x^*\|_2\le C$.
Then, $x^*$ satisfies the constraint \eqref{const:card} if $\gamma>\|\nabla f(0)\|_2+\frac{3}{2}MC$
holds.
\item[In \citep{ahn2017difference}:] Assume that $f$ is $M$-smooth and convex, $L=1,~ n_0=0,~ p=1,~ D=I$, and $c=0$.
Let $x^*$ be a d-stationary point (i.e., locally optimal solution) of \eqref{problem:GTL} and $C>0$ be a constant such that $\|x^*\|_2\le C$.
Then, $x^*$ satisfies the constraint \eqref{const:card} if $\gamma>(\sqrt{K+1}-\sqrt{K})^{-1}\left(\|\nabla f(0)\|_2+MC\right)$ holds.
\end{list}
Our exact penalty parameter $\|\nabla f(0)\|_\infty+MC$ is not only derived under weaker assumptions for each case, but also smaller than both.
Then, we also make a comparison with the results of \citet[Theorem 2]{nakayama2021superiority}\footnote{Looking at \citet{nakayama2021superiority}, it may appear at first glance that their result is a slightly stronger.
However, since their argument is a bit wrong, their result is the same one essentially.}.
\begin{list}{}{}
\item[In \citep{nakayama2021superiority}:] 
Let $(x_1^*,x_2^*)$ be an optimal solution of \eqref{problem:GTL} with $L=2,~ n_0=0,~ p_1=p_2=1,~ D_1=I,~ D_2=I,~ c_1=0$, and $c_2=0$.
Assume that $f$ is $M$-smooth, and there exist $C_1,~ C_2>0$ such that $\|x_1^*\|_2\le C_1$ and $\|x_2^*\|_2\le C_2$.
Then, $x^*$ satisfies the constraint \eqref{const:card} if $\gamma_1>\|\nabla_{x_1} f(0,0)\|_2+M\left(\frac{3}{2}C_1+C_2\right)$ and $\gamma_2>\|\nabla_{x_2} f(0,0)\|_2+M\left(C_1+\frac{3}{2}C_2\right)$ hold.
\end{list}
We see that our exact penalty parameters from Theorem \ref{thm:epp-smooth-bounded} are $\|\nabla_{x_1} f(0,0)\|_\infty+M\sqrt{C_1^2+C_2^2}$ and $\|\nabla_{x_2} f(0,0)\|_\infty+M\sqrt{C_1^2+C_2^2}$, which are smaller than $\|\nabla_{x_1} f(0,0)\|_2+M\left(C_1+C_2\right)$ and $\|\nabla_{x_2} f(0,0)\|_2+M\left(C_1+C_2\right)$, respectively, indicating a stronger result than \citet{nakayama2021superiority}.

Although Theorem \ref{thm:epp-smooth-bounded} is a stronger result than previous studies, it seems difficult to guarantee ``boundedness'' without assuming, for example, the strong convexity on $f$.
In particular, $f$ is always not strongly convex in Example \ref{ex:sparse-robust-reg}.
Motivated by this, below we derive exact penalty parameters under two more practical assumptions.
We first consider assumptions that include the case where $f$ is non-differentiable.

\begin{theorem}\label{thm:epp-Lipschitz}
Suppose that the Assumption (A1) holds, and that $f$ is directionally differentiable and Lipschitz continuous with constant $M$ under the $\ell_2$ norm.
Then, any d-stationary point $(x_0^*,x_1^*,\ldots,x_L^*)$ of \eqref{problem:GTL} satisfies the constraint \eqref{const:card} if
\begin{align}\label{ineq:epp-Lipschitz}
\gamma_l>\frac{1}{\sigma_{K_l,m_l,p_l}(D_l)}M
\end{align}
holds for all $l\in[L]$.
When $p_l=1,~ D_l=I,~ c_l=0$ for all $l\in[L]$, and $f$ is Lipschitz continuous with constant $M'$ under the $\ell_1$ norm, the right hand side of \eqref{ineq:epp-Lipschitz} can be replaced by $M'$.
\end{theorem}

\begin{proof}
By the Lipschitz continuity of $f$ and Lemma \ref{lem:d-derivative-lipschitz}, we have for all $x_1\in\mathbb{R}^{n_1},\ldots,x_L\in\mathbb{R}^{n_L}$ and $\|d_l\|_2=1$ that  
\begin{align}
f'(x_0,\ldots,x_L;0,\ldots,d_l,\ldots,0)\le M\|(0,\ldots,d_l,\ldots,0)\|_2=M,
\end{align}
which implies that the Assumption (A2) holds with $\Gamma_l=M$.
From Lemma \ref{lem:epp-GTL}, we obtain the desired result.
When $p_l=1,~ D_l=I,~ c_l=0$ for all $l\in[L]$, and $f$ is Lipschitz continuous with constant $M'$ under the $\ell_1$ norm, Lemma \ref{lem:d-derivative-lipschitz} implies that
\begin{equation}
f'(x_0,\ldots,x_L;0,\ldots,d_l,\ldots,0)\le M'\|(0,\ldots,d_l,\ldots,0)\|_1=M'
\end{equation}
for all $d_l$ such that $\|d_l\|_1=1$ and $d_l\in\{-1,0,1\}^{n_l}$.
Thus, from Lemma \ref{lem:epp-GTL}, the right hand side of \eqref{ineq:epp-Lipschitz} can be replaced by $M'$.
\end{proof}

Theorem \ref{thm:epp-Lipschitz} for $L=1,~ n_0=0,~ p=1,~ D=I$, and $c=0$ is stronger than the result by \citet[Corollary 3.4]{lu2023exact} for the unconstrained case in the sense that our result is for d-stationary points whereas their result is for global minimizers.
From Theorem \ref{thm:epp-Lipschitz}, finding the Lipschitz constant of $f$ yields exact penalty parameters for Example \ref{ex:sparse-reg-p1} (sparse linear regression with power 1), Example \ref{ex:sparse-logistic} (sparse logistic regression), and Example \ref{ex:sparse-SVM} (sparse SVM).
The proofs are straightforward and thus omitted.

\begin{corollary}[exact penalty parameter for Example \ref{ex:sparse-reg-p1}: sparse linear regression with power 1]\label{cor:sparse-reg-p1}
Let $x^*$ be a d-stationary point of \eqref{problem:sparse-reg-p1}.
If $\gamma>\max_{j\in[n]}\|a_j\|_2$ holds, then $x^*$ satisfies $T_{K,n,1}(x^*)=0$, where $a_j$ denotes the $j$th column of $A$.
\end{corollary}

\begin{corollary}[Exact penalty parameter for Example \ref{ex:sparse-logistic}: sparse logistic regression]\label{cor:sparse-logistic}
Let $x^*$ be a d-stationary point of \eqref{problem:sparse-logistic}.
If $\gamma>\sum_{j\in[q]} \|a_j\|_\infty$ holds, then $x^*$ satisfies $T_{K,n,1}(x^*)=0$.
\end{corollary}

\begin{corollary}[Exact penalty parameter for Example \ref{ex:sparse-SVM}: sparse SVM]\label{cor:sparse-SVM}
Let $x^*$ be a d-stationary point of \eqref{problem:sparse-SVM}.
If $\gamma>\sum_{j\in[q]} \|a_j\|_\infty$ holds, then $x^*$ satisfies $T_{K,n,1}(x^*)=0$.
\end{corollary}

Corollary \ref{cor:sparse-reg-p1} is equivalent to the result by \citet[Theorem 2.2]{amir2021trimmed}.
We note that it is the first time that exact penalty parameters of \eqref{problem:sparse-logistic} and \eqref{problem:sparse-SVM} are given concretely.
Next, we derive an exact penalty parameter for one of the most commonly used loss functions, which is not Lipschitz continuous.

\begin{theorem}\label{thm:epp-square}
Suppose that the Assumption (A1) holds.
\begin{enumerate}[(a)]
\item Assume that $f(x_0,x_1,\ldots,x_L)=\frac{1}{2}\|b-\sum_{l=0}^LA_lx_l\|_2^2$, where $b\in\mathbb{R}^q, A_l\in\mathbb{R}^{q\times n_l}$.
Then, $(x_0^*,x_1^*,\ldots,x_L^*)$ satisfies the constraint \eqref{const:card} if $\gamma_l>\frac{1}{\sigma_{K_l,m_l,p_l}(D_l)}\|A_l\|_2\left\|b-\sum_{l\in[L]}A_l\overline{x}_l\right\|_2$ holds for all $l\in[L]$.
\item Assume that $p_l=1,~ D_l=I,~ c_l=0$, and $f(x_1,\ldots,x_L)=\frac{1}{2}\|b-\sum_{l\in[L]}A_lx_l\|_2^2+\sum_{l\in[L]}\eta_l\|x_l\|_1$ hold, where $b\in\mathbb{R}^q, A_l\in\mathbb{R}^{q\times n_l}, \eta_l\ge0$.
Then, $(x_1^*,\ldots,x_L^*)$ satisfies the constraint \eqref{const:card} if $\gamma_l>\max_{j\in[n_l]}\|a_j^{(l)}\|_2\left\|b\right\|_2-\eta_l$ holds for all $l\in[L]$, where $a_j^{(l)}$ denotes the $j$th column of $A_l$.
\end{enumerate}
\end{theorem}

\begin{proof}
We first show the statement (a).
The equation $D_l(x_l^*-\overline{x}_l)=D_lx_l^*-c_l$ holds for any $l\in[L]$ by the Assumption (A1).
By the d-stationarity of $(x_1^*,\ldots,x_L^*)$ and Lemma \ref{lem:d-derivative-T_K} with direction $-(x_0^*,x_1^*-\overline{x}_1,\ldots,x_L^*-\overline{x}_L)$, we have
\begin{align}
\begin{split}\label{ineq:pre-descent}
&f'(x_0^*,x_1^*,\ldots,x_L^*;-x_0^*,\overline{x}_1-x_1^*,\ldots,\overline{x}_L-x_L^*)-\sum_{l=1}^L \gamma_lT_{K_l,m_l,p_l}(D_lx_l^*-c_l)\\
&=f'(x_0^*,x_1^*,\ldots,x_L^*;-x_0^*,\overline{x}_1-x_1^*,\ldots,\overline{x}_L-x_L^*)+\sum_{l=1}^L \gamma_lT'_{K_l,m_l,p_l}(D_lx_l^*-c_l;-(D_lx_l^*-c_l))\\
&=f'(x_0^*,x_1^*,\ldots,x_L^*;-x_0^*,\overline{x}_1-x_1^*,\ldots,\overline{x}_L-x_L^*)+\sum_{l=1}^L \gamma_lT'_{K_l,m_l,p_l}(D_lx_l^*-c_l;D_l(-(x_l^*-\overline{x}_l)))\\
&\ge0.
\end{split}
\end{align}
Since $f$ is convex, combining the inequality \eqref{ineq:pre-descent}, Lemma \ref{lem:d-derivative-descent-lemma}, and non-negativity of $T_{K_l,m_l,p_l}$ yields
\begin{equation}
f(0,\overline{x}_1,\ldots,\overline{x}_L)\ge f(x_0^*,x_1^*,\ldots,x_L^*),
\end{equation}
that is, $\|b-\sum_{l\in[L]}A_l\overline{x}_l\|_2\ge\|b-\sum_{l=0}^LA_lx_l^*\|_2$.
Then, for all $d_l$ such that $\|d_l\|_2=1$, we obtain
\begin{align}
f'(x_0^*,\ldots,x_L^*;0,\ldots,d_l,\ldots,0) &=\nabla_{x_l}f(x_0^*,\ldots,x_L^*)^\top d_l\\
&\le\|\nabla_{x_l}f(x_0^*,\ldots,x_L^*)\|_2\\
&=\left\|A_l^\top\left(b-\sum_{l=0}^LA_lx_l^*\right)\right\|_2\\
&\le\|A_l\|_2\left\|b-\sum_{l=0}^LA_lx_l^*\right\|_2\\
&\le\|A_l\|_2\left\|b-\sum_{l\in[L]}A_l\overline{x}_l\right\|_2,
\end{align}
where the first inequality follows from the Cauchy-Schwarz inequality.
Thus, the Assumption (A2) holds with $\Gamma_l=\|A_l\|_2\|b-\sum_{l\in[L]}A_l\overline{x}_l\|_2$.
From Lemma \ref{lem:epp-GTL}, we have the desired result.

To show the statement (b), we note that $\|b-\sum_{l=0}^LA_lx_l^*\|_2\le\|b\|_2$ can be obtained in a similar way as in the proof of statement (a).
Then, we can evaluate as
\begin{align}
\nabla_{x_l}g(x_0^*,\ldots,x_L^*)^\top d_l &\le\|\nabla_{x_l}g(x_0^*,\ldots,x_L^*)\|_\infty\\
&=\max_{j\in[n_l]}\left|{a_j^{(l)}}^\top\left(b-\sum_{l=0}^LA_lx_l^*\right)\right|\\
&\le\max_{j\in[n_l]}\|a_j^{(l)}\|_2\left\|b-\sum_{l=0}^LA_lx_l^*\right\|_2\\
&\le\max_{j\in[n_l]}\|a_j^{(l)}\|_2\left\|b\right\|_2
\end{align}
for all $d_l$ such that $\|d_l\|_1=1$ and $d_l\in\{-1,0,1\}^{n_l}$.
From Lemma \ref{lem:epp-GTL} with $\Gamma_l=\max_{j\in[n_l]}\|a_j^{(l)}\|_2\left\|b\right\|_2-\eta_l$, we have the desired result.
\end{proof}

From Theorem \ref{thm:epp-square}, we obtain exact penalty parameters for Example \ref{ex:sparse-ols} (sparse linear regression), Example \ref{ex:sparse-robust-reg} (sparse robust regression), Example \ref{ex:MCV} (minimizing constraints violation), and Example \ref{ex:trend-filtering} (piece-wise linear trend estimation).
See \ref{sec:proofs} for the proof of Corollary \ref{cor:trend-filtering}.
The proofs of the others are straightforward, and hence we omit them.

\begin{corollary}[Exact penalty parameter for Example \ref{ex:sparse-ols}: sparse linear regression]\label{cor:sparse-ols}
Let $x^*$ be a d-stationary point of \eqref{problem:sparse-ols}.
If $\gamma>\max_{j\in[n]}\|a_j\|_2\|b\|_2-\eta$ holds, then $x^*$ satisfies $T_{K,n,1}(x^*)=0$, where $a_j$ denotes the $j$th column of $A$.
\end{corollary}

\begin{corollary}[Exact penalty parameter for Example \ref{ex:sparse-robust-reg}: sparse robust regression]\label{cor:sparse-robust-reg}
Let $(x_1^*,x_2^*)$ be a d-stationary point of \eqref{problem:sparse-robust-reg}.
If $\gamma_1>\max_{j\in[n_1]}\|a_j\|_2\|b\|_2$ and $\gamma_2>\|b\|_2$ hold, then $(x_1^*,x_2^*)$ satisfies $T_{K_1,n,1}(x_1^*)=0$ and $T_{K_2,q,1}(x_2^*)=0$, where $a_j$ denotes the $j$th column of $A$.
\end{corollary}

\begin{corollary}[Exact penalty parameter for Example \ref{ex:MCV}: minimizing constraints violation]\label{cor:MCV}
Let $x^*$ be a d-stationary point of \eqref{problem:MCV}.
The point $x^*$ satisfies $T_{K,m,1}(x^*)=0$ if $\gamma>\frac{\|b-\overline{x}\|_2}{\sigma_{\min}(D)}$ holds, where $\overline{x}$ is a point that satisfies $D\overline{x}=c$.
\end{corollary}

\begin{corollary}[Exact penalty parameter for Example \ref{ex:trend-filtering}: piece-wise linear trend estimation]\label{cor:trend-filtering}
Let $x^*$ be a d-stationary point of \eqref{problem:trend-filtering},
\begin{equation}
X=
\begin{pmatrix}
1 & 1 \\
\vdots & \vdots \\
1 & n
\end{pmatrix},
\end{equation}
and $\hat{b}=X(X^\top X)^{-1}X^\top b$.
If $\gamma>\frac{\|b-\hat{b}\|_2}{2\sqrt{(1-\cos\frac{\pi}{n})(1-\cos\frac{\pi}{n-1})}}$ holds, then $x^*$ satisfies $T_{K,n-2,1}(D^{(2,n)}x^*)=0$.
\end{corollary}

Note that exact penalty parameters of \eqref{problem:sparse-robust-reg}, \eqref{problem:MCV}, and \eqref{problem:trend-filtering} are derived for the first time.
Corollary \ref{cor:sparse-ols} is equivalent to the result by \citet[Theorem 2.1]{amir2021trimmed} and slightly stronger than the one by \citet[Theorem 2.3]{bertsimas2017trimmed}\footnote{\citet{bertsimas2017trimmed} states ``$a_j$ denotes the $j$th row of $A$,'' which we think is a typo.}.
We close this section by mentioning that it is not known whether the exact penalty parameter by \citet{yagishita2024pursuit} or the exact penalty parameter derived from Lemma \ref{lem:epp-GTL} is smaller for the network trimmed lasso \citep{yagishita2024pursuit}.

\section{Trimmed Lasso with Additional Constraint}\label{sec:TL-const}
In this section, we derive exact penalty parameters for the following constrained problem:
\begin{equation}\label{problem:const-TL}
\underset{x_0\in\mathbb{R}^{n_0},x_1\in\mathbb{R}^{n_1},\ldots,x_L\in\mathbb{R}^{n_L}}{\mbox{minimize}} \quad f(x_0,x_1,\ldots,x_L)+\sum_{l\in[L]} \gamma_l T_{K_l,n_l,1}(x_l)+\delta_\mathcal{C}(x_0,x_1,\ldots,x_L),
\end{equation}
where $\mathcal{C}\subset\dom f$.
From Corollary \ref{cor:local-opt<=>d-stationary}, we have the following equivalence between the d-stationarity and the local optimality of \eqref{problem:const-TL}.

\begin{proposition}\label{prop:local-opt<=>d-stationary-in-const-TL}
D-stationary points and local minimizers of \eqref{problem:const-TL} are equivalent if $f+\delta_\mathcal{C}$ is continuous on its domain, is convex, and is directionally differentiable and $\mathcal{C}$ is closed.
\end{proposition}

The problem \eqref{problem:const-TL} includes the following examples.

\begin{example}[Sparse eigenvalue problem \citep{gotoh2018dc}]\label{ex:sparse-eigen}
For $\mathcal{C}=\{x\in\mathbb{R}^n\mid\|x\|_2\le1\}$ and $A\in\mathcal{S}^n$ such that $\lambda_{\max}(A)>0$, \citet{gotoh2018dc} considered the following problem:
\begin{equation}\label{problem:sparse-eigen}
\underset{x\in\mathbb{R}^n}{\mbox{minimize}} \quad -\frac{1}{2}x^\top Ax+\gamma T_{K,n,1}(x)+\delta_\mathcal{C}(x).
\end{equation}
\end{example}

\begin{example}[Sparse portfolio selection \citep{gotoh2018dc}]\label{ex:sparse-mean-variance}
For $\mathcal{C}=\{x\in\mathbb{R}^n\mid\mathbf{1}^\top x=1,~ x\ge0\},~ b\in\mathbb{R}^n$ and $A\in\mathcal{S}^n$ such that $A\succeq 0$, \citet{gotoh2018dc} considered the following problem:
\begin{equation}\label{problem:sparse-mean-variance}
\underset{x\in\mathbb{R}^n}{\mbox{minimize}} \quad \frac{1}{2}x^\top Ax+b^\top x+\gamma T_{K,n,1}(x)+\delta_\mathcal{C}(x).
\end{equation}
\end{example}

\begin{example}[Sparse nonnegative linear regression]\label{ex:sparse-nn-ols}
Let $b\in\mathbb{R}^q,~ A\in\mathbb{R}^{q\times n}$ be given in the same manner as in Examples \ref{ex:sparse-ols}, \ref{ex:sparse-reg-p1}, and \ref{ex:sparse-robust-reg}. For $\emptyset\neq I\subset[n]$ and $\mathcal{C}=\{x\in\mathbb{R}^n\mid x_i\ge0,~ i\in I\}$, we consider the following problem:
\begin{equation}\label{problem:sparse-nn-ols}
\underset{x\in\mathbb{R}^n}{\mbox{minimize}} \quad \frac{1}{2}\|b-Ax\|_2^2+\gamma T_{K,n,1}(x)+\delta_\mathcal{C}(x).
\end{equation}
\citet{tono2017efficient} dealt with a problem similar to this problem.
\end{example}

\begin{example}[Clustering by penalized nonnegative matrix factorization \citep{wang2021clustering}]\label{ex:NMF-clustering}
Let $A=(a_1,\ldots,a_q)\in\mathbb{R}^{p\times q}$ be a nonnegative data matrix where $a_j$ is the feature vector of the $j$th sample.
\citet{wang2021clustering} considered the following problem:
\begin{equation}\label{problem:NMF-clustering}
\underset{W\in\mathbb{R}^{p\times n},X\in\mathbb{R}^{n\times q}}{\mbox{minimize}} \quad \frac{1}{2}\|A-WX\|_F^2+\frac{\eta_1}{2}\|W\|_F^2+\frac{\eta_2}{2}\|X\|_F^2+\gamma\sum_{j\in[q]}T_{1,n,1}(x_j)+\delta_\mathcal{C}(W,X),
\end{equation}
where $\mathcal{C}=\{(W,X)\in\mathbb{R}^{p\times n}\times\mathbb{R}^{n\times q}\mid W\ge0,~X\ge0\}$, $\eta_1>0,~ \eta_2>0$, and $x_j$ is the $j$th column of $X$.
\end{example}

We note that, unlike the examples so far, the functions $f$ exploited in Examples \ref{ex:sparse-eigen} and \ref{ex:NMF-clustering} are nonconvex.
A variant of Lemma \ref{lem:epp-GTL} for the constrained problem \eqref{problem:const-TL} is given as follows.
The proof is similar to the one of the last part of Lemma \ref{lem:epp-GTL} and is therefore omitted.

\begin{lemma}\label{lem:epp-const-TL}
Let $f+\delta_\mathcal{C}$ be directionally differentiable and $x^*=(x_0^*,x_1^*,\ldots,x_L^*)$ be a d-stationary point of \eqref{problem:const-TL}.
Suppose that the following assumptions hold for $l\in[L]$:
\begin{enumerate}[({B}1)]
\item For any $x=(x_0,x_1,\ldots,x_L)\in\mathcal{C}$ and $i\in[n_l]$, it holds that $(0,\ldots,-(x_l)_ie_i,\ldots,0)\in\mathcal{F}(x;\mathcal{C})$;
\item There exists $\Gamma_l>0$ such that
\begin{align}\label{ineq:bounded-cond-const-TL}
\sup_{\substack{\|d_l\|_1=1\\ d_l\in\{-1,0,1\}^{n_l}\\ (0,\ldots,d_l,\ldots,0)\in\mathcal{F}(x^*;\mathcal{C})}}f'(x_0^*,\ldots,x_L^*;0,\ldots,d_l,\ldots,0)\le\Gamma_l.
\end{align}
\end{enumerate}
Then, $(x_0^*,x_1^*,\ldots,x_L^*)$ satisfies $T_{K_l,n_l,1}(x_l^*)=0$ if $\gamma_l>\Gamma_l$ holds.
\end{lemma}

Similar to Theorems \ref{thm:epp-smooth-bounded}--\ref{thm:epp-square}, we can obtain Theorems \ref{thm:epp-smooth-bounded-const}--\ref{thm:epp-square-const} from Lemma \ref{lem:epp-const-TL}.
Thus, we omit the proofs of them.
Exact penalty results for Examples \ref{ex:sparse-eigen} and \ref{ex:sparse-nn-ols} are immediate corollaries of Theorems \ref{thm:epp-smooth-bounded-const} (or \ref{thm:epp-Lipschitz-const}) and \ref{thm:epp-square-const}, respectively.

\begin{theorem}\label{thm:epp-smooth-bounded-const}
Suppose that the Assumption (B1) holds and $f$ is $M$-smooth on an open set $\mathcal{O}\subset\mathcal{C}$, where $\mathcal{O}$ contains the origin.
Let $(x_0^*,x_1^*,\ldots,x_L^*)$ be a d-stationary point of \eqref{problem:const-TL} and $C_l>0$ be constants such that $\|x_l^*\|_2\le C_l$ for all $l=0,\ldots,L$.
Then, $(x_0^*,x_1^*,\ldots,x_L^*)$ satisfies $\|x_l^*\|_0\le K_l$ if $\gamma_l>\|\nabla_{x_l}f(0,\ldots,0)\|_\infty+M\sqrt{\sum_{l=0}^LC_l^2}$ holds for all $l\in[L]$.
\end{theorem}

\begin{theorem}\label{thm:epp-Lipschitz-const}
Suppose that the Assumption (B1) holds, and that $f+\delta_\mathcal{C}$ is directionally differentiable and Lipschitz continuous on its domain with constant $M$ under the $\ell_1$ norm.
Then, any d-stationary point $(x_0^*,x_1^*,\ldots,x_L^*)$ of \eqref{problem:const-TL} satisfies $\|x_l^*\|_0\le K_l$ if $\gamma_l>M$ holds for all $l\in[L]$.
\end{theorem}

\begin{theorem}\label{thm:epp-square-const}
Suppose that the Assumption (B1) holds and that $-x\in\mathcal{F}(x;\mathcal{C})$ for any $x=(x_0,x_1,\ldots,x_L)\in\mathcal{C}$.
Let $f(x_0,x_1,\ldots,x_L)=\frac{1}{2}\|b-\sum_{l=0}^LA_lx_l\|_2^2$, where $b\in\mathbb{R}^q, A_l\in\mathbb{R}^{q\times n_l}$.
Then, any d-stationary point $(x_0^*,x_1^*,\ldots,x_L^*)$ of \eqref{problem:const-TL} satisfies $\|x_l^*\|_0\le K_l$ if $\gamma_l>\max_{j\in[n_l]}\|a_j^{(l)}\|_2\left\|b\right\|_2$ holds for all $l\in[L]$.
\end{theorem}

In Example \ref{ex:NMF-clustering}, although $\mathcal{C}$ satisfies Assumption (B1), the function $f$ does not satisfy any of the assumptions of Theorems \ref{thm:epp-smooth-bounded-const}--\ref{thm:epp-square-const}.
Thus, we show the following theorem.
See \ref{sec:proofs} for the proof.

\begin{theorem}\label{thm:NMF-clustering}
Let $(W^*,X^*)$ be a d-stationary point of \eqref{problem:NMF-clustering}.
If $\gamma>\max_{j\in[q]}\frac{\|a_j\|_2}{\sqrt{2}}\left(\frac{\|A\|_F}{\sqrt{\eta_1}}+\sqrt{\eta_2}\right)$ holds, then $T_{K,n,1}(x_j^*)=0$ is satisfied for all $j\in[q]$ .
\end{theorem}

The simplex $\{x\in\mathbb{R}^n\mid\mathbf{1}^\top x=1,~ x\ge0\}$ does not satisfy Assumption (B1).
However, one can derive an exact penalty parameter of it.
The proof is included in \ref{sec:proofs}.

\begin{theorem}\label{thm:sparse-port}
Let $x^*$ be a d-stationary point of \eqref{problem:const-TL} with $\mathcal{C}=\{x\in\mathbb{R}^n\mid\mathbf{1}^\top x=1,~ x\ge0\},~ L=1,~ n_0=0$, and $K\ge1$.
Suppose that $f$ is a continuous differentiable function on $\mathcal{O}\subset\mathbb{R}^n$, where $\mathcal{O}$ is an open set such that $\mathcal{C}\subset\mathcal{O}$.
It holds that $T_{K,n,1}(x^*)=0$ if $\gamma>\sqrt{2}M'$, where $M'$ is modulus of Lipschitz continuity of $f$ on $\mathcal{C}$ under the $\ell_2$ norm.
In addition to the continuous differentiability, suppose that $f$ is $M$-smooth and $0\in\mathcal{O}$.
Then, $x^*$ satisfies $T_{K,n,1}(x^*)=0$ if $\gamma>\min\left\{\sqrt{2}M',\sqrt{2}\left(\|\nabla f(0)\|_2+M\right)\right\}$ holds.
\end{theorem}

Finally, we mention the relationship with the existing results.
The exact penalty parameters of the problem \eqref{problem:const-TL} presented by \citet{gotoh2018dc} are applicable only at globally optimal solutions.
As in the previous section, our exact penalty parameters are stronger in two ways, namely, ours are not only valid at d-stationary points, but also the thresholds are no greater than the existing ones even under weaker assumptions\footnote{The exact penalty parameter by \citet[Corollary 3]{gotoh2018dc} is written as $\sqrt{2}\|\nabla f(0)\|_2-(\sqrt{2}+1)M$, but there is a typo; the correct one is $\sqrt{2}\|\nabla f(0)\|_2+(\sqrt{2}+1)M$.}. 
\citet{wang2021clustering} derived an exact penalty parameter of \eqref{problem:NMF-clustering} at d-stationary points, whereas boundedness of $(W^*,X^*)$ is imposed a priori.
We succeed in removing the boundedness assumption in Theorem \ref{thm:NMF-clustering}.

\section{Penalized Problem by Truncated Nuclear Norm}\label{sec:truncated-nuclear}
In this section, we consider parallel results to the previous sections for the rank constraint.

\subsection{Truncated Nuclear Norm and Rank-Constrained Formulation}
For a matrix $X\in\mathbb{R}^{m\times n}$, the {\it truncated nuclear norm} is defined by
\begin{equation}
\mathcal{T}_K(X)\coloneqq\sum_{i=K+1}^q\sigma_i(X),
\end{equation}
where $K\in\{0,\ldots,q-1\}$ and $q\coloneqq\min\{m,n\}$.
It holds that $\mathcal{T}_K(X)\ge0$ for any $X$, and $\mathcal{T}_K(X)=0$ if and only if
\begin{equation}\label{ineq:simple-rank}
\mathrm{rank}(X)\le K.
\end{equation}
Thus, we can view that the truncated nuclear norm measures the violation of the rank constraint \eqref{ineq:simple-rank}, which was first pointed out by \citet{gao2010majorized}.
Obviously, $\mathcal{T}_K$ is reduced to the nuclear norm when $K=0$.
Note also that $\mathcal{T}_K$ is directionally differentiable since $\mathcal{T}_K$ is the difference of the nuclear norm and the Ky Fan $K$ norm \citep{gotoh2018dc}.
We consider the following minimization problem:
\begin{equation}\label{problem:truncated-nuclear}
\underset{X\in\mathbb{R}^{m\times n}}{\mbox{minimize}} \quad f(X)+\gamma\mathcal{T}_K(X)+\delta_\mathcal{C}(X),
\end{equation}
where $\gamma>0,~ \mathcal{C}\subset\dom f$.
This formulation involves a lot of existing problems penalized by the truncated nuclear norm \citep{gao2010majorized,zhang2012matrix,hu2012fast,hu2015large,oh2015partial,hong2016online,zhang2016robust,bi2016error,gotoh2018dc,liu2020exact,qian2023calmness,lu2023exact}.
To the best of our knowledge, the first penalized problems via the truncated nuclear norm were proposed independently by \citet{gao2010majorized} (with some constraints) and \citet{zhang2012matrix}.
As is the case for the generalized trimmed lasso \eqref{problem:GTL}, any optimal (resp. locally optimal, d-stationary) solution of \eqref{problem:truncated-nuclear} satisfying the constraint \eqref{ineq:simple-rank} (or $\mathcal{T}_K(X)=0$) is optimal (resp. locally optimal, d-stationary) to the problem
\begin{align}
\underset{X\in\mathbb{R}^{m\times n}}{\mbox{minimize}} & \quad f(X)+\delta_\mathcal{C}(X)\\
\text{subject to}  & \quad \mathrm{rank}(X)\le K.
\end{align}
The following remarkable applications are included in the problem \eqref{problem:truncated-nuclear}.

\begin{example}[Small rank linear regression \citep{zhang2012matrix,hu2012fast,gotoh2018dc}]\label{ex:small-rank-ols}
\citet{gotoh2018dc} considered the following problem:
\begin{equation}\label{problem:small-rank-ols}
\underset{X\in\mathbb{R}^{m\times n}}{\mbox{minimize}} \quad \frac{1}{2}\|\mathcal{A}(X)-b\|_2^2+\gamma\mathcal{T}_K(X),
\end{equation}
where $b\in\mathbb{R}^p,~ A_1,\ldots,A_p\in\mathbb{R}^{m\times n}$ and $\mathcal{A}(X)=(A_1\bullet X,\ldots,A_p\bullet X)^\top$.
\citet{zhang2012matrix} and \citet{hu2012fast} proposed the matrix completion problem that is the special case of this problem with $b=(M_{i,j})_{(i,j)\in\Omega}$ and $\mathcal{A}(X)=(X_{i,j})_{(i,j)\in\Omega}$, where $M$ is a given $m\times n$ matrix and $\Omega\subset[m]\times[n]$.
\end{example}

\begin{example}[Multi-class classification \citep{hu2015large}]\label{ex:multi-class}
Let $b_1,\ldots,b_q\in[n]$ and $a_1,\ldots,a_q\in\mathbb{R}^m$ denote, respectively, the class labels and features vectors of $q$ samples. To find $n$-class classifiers, \citet{hu2015large} considered the following problem:
\begin{equation}\label{problem:multi-class}
\underset{X\in\mathbb{R}^{m\times n}}{\mbox{minimize}} \quad \sum_{j\in[q]}\log\Big(1+\sum_{l\in[n]\setminus\{b_j\}}\exp(a_j^\top x_l-a_j^\top x_{b_j})\Big)+\gamma\mathcal{T}_K(X),
\end{equation}
where $x_l$ denotes the $l$th column of $X$.
\end{example}

\begin{example}[Robust principal component analysis \citep{oh2015partial,zhang2016robust}]\label{ex:robust-PCA}
\citet{oh2015partial} and \citet{zhang2016robust} considered the following problem:
\begin{equation}\label{problem:robust-PCA}
\underset{X\in\mathbb{R}^{m\times n}}{\mbox{minimize}} \quad \|A-X\|_1+\gamma\mathcal{T}_K(X),
\end{equation}
where $A\in\mathbb{R}^{m\times n}$.
\end{example}

\begin{example}[Semidefinite box constrained problem with quadratic objective \citep{liu2020exact}]\label{ex:sdb-const-quad}
\citet{liu2020exact} considered the following problem:
\begin{equation}\label{problem:sdb-const-quad}
\underset{X\in\mathbb{R}^{n\times n}}{\mbox{minimize}} \quad \frac{1}{2}\|H\circ(X-A)\|_F^2+\gamma\mathcal{T}_K(X)+\delta_\mathcal{C}(X),
\end{equation}
where $\mathcal{C}=\{X\in\mathcal{S}^n\mid X\succeq 0,~ I-X\succeq 0\},~ A\in\mathcal{S}^n$, and $H\in\mathcal{S}^n$ is a weight matrix whose entries are nonnegative.
The spherical sensor localization problem and the nearest low-rank correlation problem are included in this problem (see \citet[Section 6]{liu2020exact}).
\end{example}

We remark that unlike the generalized trimmed lasso \eqref{problem:GTL} the equivalence of the local optimality and the d-stationarity does not hold in the problem \eqref{problem:truncated-nuclear} even when $f$ is a convex function.
To see this, let us consider the special case of \eqref{problem:truncated-nuclear}
\begin{align}\label{problem:counterex-truncated-nuclear}
\underset{X\in\mathbb{R}^{2\times2}}{\mbox{minimize}} & \quad f(X)+\gamma\mathcal{T}_1(X)=\underbrace{ \begin{psmallmatrix}
-\gamma & 0 \\
0 & 0
\end{psmallmatrix}
\bullet X}_{f(X)}+\gamma\underbrace{\sigma_2(X)}_{\mathcal{T}_1(X)}.
\end{align}
For $X^*=\begin{psmallmatrix}
1 & 0 \\
0 & 2
\end{psmallmatrix}
$, we have
\begin{equation}
\sigma_2(X^*+M)=\sigma_2(X^*)+
\begin{psmallmatrix}
1 & 0 \\
0 & 0
\end{psmallmatrix}
\bullet M+O(\|M\|_F^2)
\end{equation}
from the result of \citet[Subsection 5.1]{lewis2005nonsmooth}.
This implies that $\sigma_2$ is differentiable at the point $X^*$ and the gradient is $\nabla\sigma_2(X^*)=\begin{psmallmatrix}
1 & 0 \\
0 & 0
\end{psmallmatrix}$, so $X^*$ is a d-stationary point of \eqref{problem:counterex-truncated-nuclear}.
On the other hand, with the direction $Y=\begin{psmallmatrix}
2 & -1 \\
-1 & 2
\end{psmallmatrix}$ and a step size $\xi>0$, consider the perturbed value 
\begin{align}
h(\xi)\coloneqq &f(X^*+\xi Y)+\gamma\sigma_2(X^*+\xi Y)\\
= &-\gamma-2\gamma\xi+\frac{\gamma}{2}(3+4\xi-\sqrt{4\xi^2+1})\\
= &\frac{\gamma}{2}(1-\sqrt{4\xi^2+1}),
\end{align}
where the second equality follows from a simple eigenvalue calculation.
Since $h'(\xi)=-\frac{2\gamma\xi}{\sqrt{4\xi^2+1}}<0$, $X^*$ is not a local minimizer of \eqref{problem:counterex-truncated-nuclear}.
We also see by Corollary \ref{cor:local-opt<=>d-stationary} that $\mathcal{T}_1:\mathbb{R}^{2\times2}\to\mathbb{R}$ cannot be represented as the pointwise minimum of a finite number of convex functions.

\subsection{Exact Penalties for Problems with Truncated Nuclear Norm}
We will derive exact penalty parameters at d-stationary points of the truncated nuclear penalized problem \eqref{problem:truncated-nuclear} in the remainder of this section.
The only existing exact penalty parameters of the problem \eqref{problem:truncated-nuclear} are for globally optimal solutions \citep{bi2016error,gotoh2018dc,liu2020exact,qian2023calmness,lu2023exact}.
Let us start with the following lemma, which is an analogue of Lemma \ref{lem:epp-GTL}.
Without loss of generality, in the following we assume that $q=m\le n$.

\begin{lemma}\label{lem:epp-truncated-nuclear}
Let $f+\delta_\mathcal{C}$ be directionally differentiable and $X^*$ be a d-stationary point of \eqref{problem:truncated-nuclear}.
Suppose that the following assumptions hold:
\begin{enumerate}[({C}1)]
\item For any $X\in\mathcal{C}$, there is a singular value decomposition $X=U\big[\mathrm{diag}\big(\sigma_1(X),\ldots,\sigma_q(X)\big),0\big]V^\top$ such that $-U\big[\mathrm{diag}\big(0,\ldots,0,\sigma_{K+1}(X),\ldots,\sigma_q(X)\big),0\big]V^\top\in\mathcal{F}(X;\mathcal{C})$;
\item There exists $\Gamma$ such that
\begin{align}\label{ineq:bounded-truncated-nuclear}
\sup_{\substack{\|D\|_*=1\\ D\in\mathcal{F}(X^*;\mathcal{C})}}f'(X^*;D)\le\Gamma.
\end{align}
\end{enumerate}
Then, the point $X^*$ satisfies $\mathcal{T}_K(X^*)=0$ if $\gamma>\Gamma$ holds.
\end{lemma}

\begin{proof}
To derive a contradiction, we assume that $\mathcal{T}_K(X^*)>0$.
It holds that
\begin{equation}\label{ineq:stat-truncated-nuclear}
f'(X^*;D)+\gamma\mathcal{T}_K'(X^*;D)\ge0
\end{equation}
for any $D\in\mathcal{F}(X^*;\mathcal{C})$ because $X^*$ is a d-stationary point. 
Let $X^*=U\Sigma V^\top$ be a singular value decomposition of $X^*$ in Assumption (C1), that is, $-U\big[\mathrm{diag}\big(0,\ldots,0,\sigma_{K+1}(X^*),\ldots,\sigma_q(X^*)\big),0\big]V^\top\in\mathcal{F}(X^*;\mathcal{C})$.
Letting $\Sigma'=\big[\mathrm{diag}\big(0,\ldots,0,\sigma_{K+1}(X^*),\ldots,\sigma_q(X^*)\big),0\big]$, we see that
\begin{align}\label{ineq:d-derivative-tnn}
\begin{split}
\lefteqn{\mathcal{T}_K'(X^*;-U\Sigma'V^\top)}\\
&=\lim_{\eta\searrow0}\frac{\mathcal{T}_K(X^*-\eta U\Sigma'V^\top)-\mathcal{T}_K(X^*)}{\eta}\\
&=\lim_{\eta\searrow0}\frac{\mathcal{T}_K(U(\Sigma-\eta\Sigma')V^\top)-\mathcal{T}_K(X^*)}{\eta}\\
&=\lim_{\eta\searrow0}\frac{\mathcal{T}_K(U[\mathrm{diag}(\sigma_1(X^*),\ldots,\sigma_K(X^*),(1-\eta)\sigma_{K+1}(X^*),\ldots,(1-\eta)\sigma_q(X^*)),0]V^\top)-\mathcal{T}_K(X^*)}{\eta}\\
&=\lim_{\eta\searrow0}\frac{(1-\eta)\sum_{K+1}^q\sigma_i(X^*)-\mathcal{T}_K(X^*)}{\eta}\\
&=\lim_{\eta\searrow0}\frac{(1-\eta)\mathcal{T}_K(X^*)-\mathcal{T}_K(X^*)}{\eta}\\
&=\lim_{\eta\searrow0}\frac{-\eta\mathcal{T}_K(X^*)}{\eta}\\
&=-\mathcal{T}_K(X^*).
\end{split}
\end{align}
Note that $-U\Sigma'V^\top\neq0$ holds because the inequality $\mathcal{T}_K(X^*)>0$ is assumed.
Substituting $D=-U\Sigma'V^\top$ into \eqref{ineq:stat-truncated-nuclear} yields
\begin{equation}
f'(X^*;-U\Sigma'V^\top)\ge\gamma\mathcal{T}_K(X^*)=\gamma\sum_{i=K+1}^q\sigma_i(X^*)=\gamma\|-U\Sigma'V^\top\|_*.
\end{equation}
By the positive homogeneity of directional derivatives with respect to the direction and the assumption \eqref{ineq:bounded-truncated-nuclear}, we have
\begin{equation}
\gamma\le f'\left(X^*;\frac{-U\Sigma'V^\top}{\|-U\Sigma'V^\top\|_*}\right)\le\Gamma,
\end{equation}
which contradicts the inequality $\gamma>\Gamma$.
\end{proof}

Similar to the generalized trimmed lasso \eqref{problem:GTL} and the constrained trimmed lasso \eqref{problem:const-TL}, we see that ``the boundedness assumption for the gradient of $f$ at stationary points'' is essential in the problem \eqref{problem:truncated-nuclear}.
Using Lemma \ref{lem:epp-truncated-nuclear}, we first derive the following exact penalty parameter of the truncated nuclear penalized problem \eqref{problem:truncated-nuclear}.
See \ref{sec:proofs} for the proof.

\begin{theorem}\label{thm:epp-smooth-bounded-tn}
Suppose that the Assumption (C1) holds and $f$ is $M$-smooth on an open set $\mathcal{O}\subset\mathcal{C}$, where $\mathcal{O}$ contains the origin.
Let $X^*$ be a d-stationary point of \eqref{problem:truncated-nuclear} and $C>0$ be a constant such that $\|X^*\|_F\le C$.
Then, $X^*$ satisfies the constraint $\mathcal{T}_K(X^*)=0$ if $\gamma>\|\nabla f(0)\|_2+MC$ holds.
\end{theorem}

We compare Theorem \ref{thm:epp-smooth-bounded-tn} with the result of \citet[Theorem 5]{gotoh2018dc}.
\begin{list}{}{}
\item[In \citep{gotoh2018dc}:] Assume that $f$ is $M$-smooth.
Let $X^*$ be an optimal solution of \eqref{problem:truncated-nuclear} and $C>0$ be a constant such that $\|X^*\|_F\le C$.
Then, $X^*$ satisfies the constraint $\mathcal{T}_K(X^*)=0$ if $\gamma>\|\nabla f(0)\|_F+\frac{3}{2}MC$
holds.
\end{list}
Theorem \ref{thm:epp-smooth-bounded-tn} is stronger than Theorem 5 of \citet{gotoh2018dc} in two ways: (i) our exact penalty parameter $\|\nabla f(0)\|_2+MC$ is derived at d-stationary points, which is easier to attain than global minimizers; (ii) ours is smaller than theirs. 
However, the boundedness assumption of the solution set is impractical in a more plausible way than the generalized trimmed lasso \eqref{problem:GTL}.
For example, for the matrix completion problem, which is a special case of Example \ref{ex:small-rank-ols}, $f$ is not strongly convex whenever $\Omega\subsetneq[m]\times[n]$ holds.
Therefore, in the following, we derive exact penalty parameters under realistic assumptions.
Theorems \ref{thm:epp-Lipschitz-tn} and \ref{thm:epp-small-rank-ols} are parallel to Theorems \ref{thm:epp-Lipschitz} and \ref{thm:epp-square}, respectively.
See \ref{sec:proofs} for the proofs.

\begin{theorem}\label{thm:epp-Lipschitz-tn}
Suppose that the Assumption (C1) holds, and that $f+\delta_\mathcal{C}$ is directionally differentiable and Lipschitz continuous on its domain with constant $M$ under the nuclear norm.
Then, any d-stationary point $X^*$ of \eqref{problem:truncated-nuclear} satisfies $\mathcal{T}_K(X^*)=0$ if $\gamma>M$ holds.
\end{theorem}

Theorem \ref{thm:epp-Lipschitz-tn} is stronger than the result by \citet[Corollary 4.7]{lu2023exact} for the unconstrained case in the sense that our result is for d-stationary points whereas their result is for global minimizers. 
Since Theorem \ref{thm:epp-Lipschitz-tn} does not require the boundedness of the set of optimal solutions but the boundedness of gradients of the first term $f$, concrete exact penalty parameters for Examples \ref{ex:multi-class} (multi-class classification) and Example \ref{ex:robust-PCA} (robust PCA) can be derived as immediate corollaries of the theorem. 

\begin{corollary}[Exact penalty parameter for Example \ref{ex:multi-class}: multi-class classification]\label{cor:multi-class}
Let $X^*$ be a d-stationary point of \eqref{problem:multi-class}.
If $\gamma>\sqrt{2}\sum_{j\in[q]} \|a_j\|_2$ holds, then $X^*$ satisfies $\mathcal{T}_K(X^*)=0$.
\end{corollary}

\begin{corollary}[Exact penalty parameter for Example \ref{ex:robust-PCA}: robust PCA]\label{cor:robust-PCA}
Let $X^*$ be a d-stationary point of \eqref{problem:robust-PCA}.
If $\gamma>\sqrt{mn}$ holds, then $X^*$ satisfies $\mathcal{T}_K(X^*)=0$.
\end{corollary}

As for Example \ref{ex:small-rank-ols} (small-rank linear regression), the first term is not Lipschitz continuous, so Theorem \ref{thm:epp-Lipschitz-tn} can not be applied.
However, we can obtain an exact penalty result that is applicable for Example \ref{ex:small-rank-ols}.

\begin{theorem}\label{thm:epp-small-rank-ols}
Suppose that the Assumption (C1) holds and that $-X\in\mathcal{F}(X;\mathcal{C})$ for any $X\in\mathcal{C}$.
Let $f(X)=\frac{1}{2}\|\mathcal{A}(X)-b\|_2^2$, where $b\in\mathbb{R}^p,~ A_1,\ldots,A_p\in\mathbb{R}^{m\times n}$ and $\mathcal{A}(X)=(A_1\bullet X,\ldots,A_p\bullet X)^\top$.
Any d-stationary point $X^*$ of \eqref{problem:truncated-nuclear} satisfies $\mathcal{T}_K(X^*)=0$ if it is holds
\begin{equation}
\gamma>\|\mathcal{A}^*\|_{2,2}\|b\|_2,
\end{equation}
where $\mathcal{A}^*$ is the adjoint operator of $\mathcal{A}$ and $\|\mathcal{A}^*\|_{2,2}=\sup_{\|y\|_2=1}\|\mathcal{A}^*(y)\|_2$.
\end{theorem}

Finally, we note that, by using Theorems \ref{thm:epp-smooth-bounded-tn} and \ref{thm:epp-Lipschitz-tn}, one can obtain a stronger exact penalty result than \citet[Theorem 3.1]{liu2020exact} for the problem \eqref{problem:truncated-nuclear} with $\mathcal{C}=\{X\in\mathcal{S}^n\mid X\succeq 0,~ CI-X\succeq 0\}$.

\section{Discussions related to algorithms}\label{sec:discussions}
In the previous sections, we have shown that under the mild assumptions any d-stationary point satisfies the cardinality or rank constraint if $\gamma$ is taken to be large enough.
To obtain a d-stationary point, the key point is that the proximal mappings of the trimmed $\ell_1$ norm and truncated nuclear norm is explicitly represented, though they are nonconvex and non-separable unlike $\ell_1$ norm.
The proximal mapping of $g$ is defined as follows:
\begin{align}
\mathrm{prox}_g(x)\coloneqq\argmin_{z\in\mathbb{E}}\Big\{g(z)+\frac{1}{2}\|z-x\|_2^2\Big\}.
\end{align}
In fact, the proximal mappings of $\gamma'T_{K,m,p}$, $\gamma'\mathcal{T}_K$, and $\gamma'T_{K,n,1}+\eta'\|\cdot\|_1$ have been derived by \citet[Subsection 4.2]{yagishita2024pursuit}, \citet[Theorem 1]{oh2015partial}, and \citet[Proposition 2.4]{luo2013new}, respectively.
\citet{banjac2017novel} and \citet[Theorem 4.4]{liu2020exact} have also derived the proximal mapping of $\gamma'T_{K,n,1}+\delta_{[-C,C]^n}$, $\gamma'T_{K,n,1}+\delta_{[0,C]^n}$, and $\gamma'\mathcal{T}_K+\delta_{\mathcal{C}}$, respectively, where $C>0$ and $\mathcal{C}=\{X\in\mathcal{S}^n\mid X\succeq 0,~ I-X\succeq 0\}$.
These explicit formulas of the proximal mappings can be used to construct efficient algorithms for obtaining d-stationary points.

As proximal map-based algorithms for solving the generalized trimmed lasso \eqref{problem:GTL}, the constrained trimmed lasso \eqref{problem:const-TL}, and the truncated nuclear penalized problem \eqref{problem:truncated-nuclear}, the {\it proximal alternating direction method of multipliers} (proximal ADMM) \citep{li2015global} and the {\it proximal gradient method} (PGM) are considered. 
\citet{bertsimas2017trimmed} and \citet{yagishita2024pursuit} used the proximal ADMM for solving special cases of the generalized trimmed lasso \eqref{problem:GTL}.
\citet{luo2013new} and \citet{huang2015two} solved the problem \eqref{problem:sparse-ols} by the PGM.
For solving special cases of the generalized trimmed lasso \eqref{problem:GTL} and the problem \eqref{problem:truncated-nuclear} with $\mathcal{C}=\{X\in\mathcal{S}^n\mid X\succeq 0,~ I-X\succeq 0\}$, \citet{lu2018sparse}, \citet{nakayama2021superiority}, and \citet{liu2020exact} used the SpaRSA \citep{wright2009sparse,gong2013general}, which is a variant of the PGM exploiting a non-monotone linear search \citep{grippo1986nonmonotone} and the Barzilai-Borwein rule \citep{barzilai1988two}.
The global convergence result of the proximal ADMM to a d-stationary point for the penalized problem in this paper is given in the same as Subsection 4.4 of \citet{yagishita2024pursuit}.
For the PGM, the global convergence to a d-stationary point was given by \citet[Theorem 1]{razaviyayn2013unified} and \citet[Theorem 1]{nakayama2021superiority}.
\citet{liu2020exact} showed an exact penalty result for global minimizers, and that the SpaRSA converges to a l(imiting)-stationary point, which is a weaker stationarity notion than d-stationarity \citep{cui2018composite}, of the problem \eqref{problem:truncated-nuclear} with $f$ is $M$-smooth and $\mathcal{C}=\{X\in\mathcal{S}^n\mid X\succeq 0,~ I-X\succeq 0\}$. 
While there was an inconsistency between the exact penalty result and the convergence result in \citet{liu2020exact}, our exact penalty results for d-stationarity fill the gap. 
The same is true for the inconsistency in \citet{nakayama2021superiority}.

Besides the proximal ADMM and the PGM, it is known that the EDCA \citep{pang2017computing}, the EPDCA \citep{lu2019enhanced}, and the NEPDCA \citep{lu2019nonmonotone} can also obtain a d-stationary point (see those papers for the details).
However, while those algorithms are also applicable to the problems penalized by the trimmed $\ell_1$ norm with $p=1$, it seems difficult to apply to the generalized trimmed lasso \eqref{problem:GTL} with $p_l>1$ for some $l\in[L]$ and the problems with the truncated nuclear norm penalty \eqref{problem:truncated-nuclear}.
A homotopy method proposed by \citet{amir2021trimmed} for solving the problem \eqref{problem:sparse-ols} can be extended to for the generalized trimmed lasso \eqref{problem:GTL}.
However, this algorithm has only yielded convergence results to a d-stationary point or ambiguous point.\footnote{For the generalized trimmed lasso \eqref{problem:GTL}, a point $(x_1,\ldots,x_L)$ is called ambiguous point if $\|(D_lx_l-c_l)_{(K)}\|_2=\|(D_lx_l-c_l)_{(K+1)}\|_2$ holds for some $l\in[L]$.}

Since the trimmed $\ell_1$ norm and the truncated nuclear norm are represented by the difference of two convex (DC) functions (see, e.g., \citet{gotoh2018dc}), one might consider applying the DC algorithm (DCA) and then obtaining a critical point, which is also a weaker stationarity notion than d-stationarity \citep{cui2018composite}.
However, in the context of the exact penalty, which often take a large penalty parameter, the use of the DCA is not preferred.
To see this, let us consider the generalized trimmed lasso \eqref{problem:GTL} with $L=1,~ n_0=0,~ p=1,~ D=I,~ c=0,~ K\ge1$ and suppose also that $f$ is a differentiable convex function.
By using the largest-$K$ norm
\begin{equation}
    \|x\|_{\langle K\rangle}\coloneqq\max_{\substack{\Lambda\subset[n]\\|\Lambda|=K}}\sum_{i\in\Lambda}|x_i|,
\end{equation}
the trimmed $\ell_1$ norm is represented as $T_{K,n,1}(x)=\|x\|_1-\|x\|_{\langle K\rangle}$.
We call $x^*\in\mathbb{R}^n$ a critical point of \eqref{problem:GTL} if it holds that
\begin{equation}
    0\in\nabla f(x^*)+\gamma\partial\|x^*\|_1-\gamma\partial\|x^*\|_{\langle K\rangle},
\end{equation}
where $\partial\|x^*\|_1$ and $\partial\|x^*\|_{\langle K\rangle}$ are the subdifferential sets of the $\ell_1$ norm and the largest-$K$ norm at $x^*$, respectively.
As the subdifferential of a norm at $x=0$ is its dual norm unit ball (see, e.g., Example 3.3 of \citet{beck2017first}), unfortunately, the origin is always a critical point of \eqref{problem:GTL} provided that $\gamma\ge\|\nabla f(0)\|_\infty$.
On the other hand, since $T_K'(0;\pm e_i)=0$ for any $i\in[n]$, the origin is never a d-stationary point of \eqref{problem:GTL} for any $\gamma>0$ unless $\nabla f(0)=0$.

\section{Concluding Remarks}\label{sec:conclusion}
In this paper, we have extended the trimmed lasso to the generalized trimmed lasso and shown results of exact penalization for the generalized trimmed lasso and the problem penalized by the truncated nuclear norm under weaker assumptions than in the existing studies.
The unified analysis for the generalized trimmed lasso reveals the essential assumption on the loss function term for the existence of the exact penalty parameter.
In particular, to the best of our knowledge, no exact penalty results at other than globally optimal solutions have been known for the truncated nuclear norm.
As a result, we have seen that the existing local search algorithms such as the proximal ADMM and proximal gradient algorithms yield a point that satisfies the (structured) sparsity or rank constraint. 
On the other hand, we have not been able to derive an exact penalty parameter at d-stationary points for penalized problems that have not-simple constraints such as the one addressed in \citet{gao2010majorized}.
We left the derivation of the exact penalty parameters at stationary points of the constrained problems for future works, although \citet{bi2016error} and \citet{qian2023calmness} have derived exact penalties at globally optimal solutions of some of them.
We have also shown the relationship between local optimality and d-stationarity of these problems.
By the equivalence result, d-stationary points of the generalized trimmed lasso that are attainable by the existing algorithms are also locally optimal if the loss function $f$ is convex. 
Furthermore, our equivalence result between local minimizers and d-stationary points is valid also for the more general problem whose objective function is defined by the pointwise minimum of finitely many convex functions, so it might be useful for formulations in other contexts.

\section*{Acknowledgments}
This work was supported in part by the Research Institute for Mathematical Sciences, an International Joint Usage/Research Center located in Kyoto University.
Jun-ya Gotoh is supported in part by JSPS KAKENHI Grant 20H00285 and 24K01113.

\appendix
\def\thesection{Appendix \Alph{section}}

\section{Proofs}\label{sec:proofs}
In this section, we provide the proofs omitted in the main body of this paper.

\subsection{Proof of Corollary \ref{cor:trend-filtering}}
\begin{proof}
Considering that $D^{(2,n)}$ is the second-order difference matrix, we see that there is equivalence between $D^{(2,n)}\overline{x}=0$ being satisfied and $\overline{x}$ being a linear time series.
The minimizer of the optimization problem
\begin{align}
\underset{\overline{x}\in\mathbb{R}^{n}}{\mbox{minimize}} & \qquad \|b-\overline{x}\|_2 \\
\text{subject to} & \qquad D^{(2,n)}\overline{x}=0
\end{align}
is $\hat{b}=X(X^\top X)^{-1}X^\top b$, where
\begin{equation}
X=
\begin{pmatrix}
1 & 1 \\
\vdots & \vdots \\
1 & n
\end{pmatrix}.
\end{equation}
Besides, since the second-order difference matrix $D^{(2,n)}$ can be expressed as $D^{(2,n)}=D^{(1,n-1)}D^{(1,n)}$, using the first-order difference matrix
\begin{equation}
D^{(1,n)}\coloneqq
\begin{pmatrix}
1 & -1 & & & \\
& 1 & -1 & & \\
& & \ddots & \ddots & \\
& & & 1 & -1 \\
\end{pmatrix}
\in\mathbb{R}^{(n-1)\times n},
\end{equation}
we have
\begin{equation}
\left\|{D^{(2,n)}}^\top y\right\|_2^2=\left\|{D^{(1,n)}}^\top{D^{(1,n-1)}}^\top y\right\|_2^2\ge\lambda_{\min}(D^{(1,n)}{D^{(1,n)}}^\top)\lambda_{\min}(D^{(1,n-1)}
{D^{(1,n-1)}}^\top)\|y\|_2^2
\end{equation}
for all $y\in\mathbb{R}^{n-2}=\ker({D^{(2,n)}}^\top)^\perp$.
Using $\lambda_{\min}(D^{(1,n)}{D^{(1,n)}}^\top)=2(1-\cos\frac{\pi}{n})$ \citep[Theorem 2.2]{kulkarni1999eigenvalues} and Lemma \ref{lem:restricted bounded below}, we obtain $\sigma_{\min}(D^{(2,n)})\ge2\textstyle\sqrt{(1-\cos\frac{\pi}{n})(1-\cos\frac{\pi}{n-1})}$.
Since Assumption (A1) is satisfied with $\overline{x}=\hat{b}$, from Theorem \ref{thm:epp-square}, we have the desired result.
\end{proof}

\subsection{Proof of Theorem \ref{thm:NMF-clustering}}
\begin{proof}
Since $(-W^*,0)\in\mathcal{F}((W^*,X^*);\mathcal{C})$, we see from the d-stationarity of $(W^*,X^*)$ that
\begin{equation}
\nabla_Wf(W^*,X^*)\bullet(-W^*)\ge0.
\end{equation}
As $f$ is $\eta_1$-strongly convex with respect to $W$, we have
\begin{equation}
f(0,X^*)-f(W^*,X^*)-\frac{\eta_1}{2}\|W^*\|_F^2\ge0.
\end{equation}
Accordingly, it holds that $\|W^*\|_F^2\le\frac{1}{2\eta_1}\|A\|_F^2$.
Besides, for any $j\in[q]$, since $(0,(0,\ldots,-x_j^*,\ldots,0))\in\mathcal{F}((W^*,X^*);\mathcal{C})$ and $T_{1,n,1}'(x_j^*;-x_j^*)=-T_{1,n,1}(x_j^*)$, we see again from the d-stationarity of $(W^*,X^*)$ that
\begin{equation}
\nabla_{x_j}f(W^*,X^*)^\top(-x_j^*)-T_{1,n,1}(x_j^*)\ge0.
\end{equation}
As $f$ is $\eta_2$-strongly convex with respect to $x_j$ and $T_{1,n,1}$ is nonnegative, we obtain
\begin{equation}
\eta_2\|x_j^*\|_2^2+\frac{1}{2}\|a_j-W^*x_j^*\|_2^2\le\frac{1}{2}\|a_j\|_2^2.
\end{equation}
Then, for all $d_j$ such that $\|d_j\|_1$ and $d_j\in\{-1,0,1\}^n$, we can evaluate as
\begin{align}
\nabla_{x_j}f(W^*,X^*)^\top d_j &\le\|{W^*}^\top(W^*x_j^*-a_j)+\eta_2x_j^*\|_2\\
&\le\|W^*\|_F\|(W^*x_j^*-a_j)\|_2+\eta_2\|x_j^*\|_2\\
&\le\frac{\|a_j\|_2}{\sqrt{2}}\left(\frac{\|A\|_F}{\sqrt{\eta_1}}+\sqrt{\eta_2}\right).
\end{align}
From Lemma \ref{lem:epp-const-TL}, we have the desired result.
\end{proof}

\subsection{Proof of Theorem \ref{thm:sparse-port}}
\begin{proof}
We will derive a contradiction by assuming that $T_{K,n,1}(x^*)>0$.
From Lemma \ref{lem:d-derivative-T_K} and the d-stationarity of $x^*$, it holds that
\begin{equation}
f'(x^*;d)+\gamma\sum_{i\in\Lambda_1\cup\Lambda}\Delta(x_i^*;d_i)\ge0
\end{equation}
for some $\Lambda\subset\Lambda_2$ such that $|\Lambda|=n-K-|\Lambda_1|$ and any $d\in\mathcal{F}(x^*;\mathcal{C})$.
We note that there exists $i'\in\Lambda_1\cup\Lambda$ such that $x_{i'}^*\neq0$ because $T_{K,n,1}(x^*)>0$ is assumed, and we set $d'=x_{i'}^*(e_{i''}-e_{i'})\in\mathcal{F}(x^*;\mathcal{C})$ for some $i''\notin\Lambda_1\cup\Lambda$.
The existence of such an $i''$ is guaranteed by $K\ge1$.
Substituting $d=d'$, we have
\begin{equation}\label{ineq:d-stat-sparse-port}
f'(x^*;x_{i'}^*(e_{i''}-e_{i'})) \ge\gamma|x_{i'}^*|.
\end{equation}
Form Lemma \ref{lem:d-derivative-lipschitz}, the left hand side is bounded by
\begin{equation}
f'(x^*;x_{i'}^*(e_{i''}-e_{i'}))\le M'\|x_{i'}^*(e_{i''}-e_{i'})\|_2=\sqrt{2}M'|x_{i'}^*|.
\end{equation}
Combining this and \eqref{ineq:d-stat-sparse-port} yields $\gamma\le\sqrt{2}M'$, which is the contradiction to the inequality $\gamma>\sqrt{2}M'$.

Suppose also that $f$ is $M$-smooth and $0\in\mathcal{O}$, then we have
\begin{align}
f'(x^*;x_{i'}^*(e_{i''}-e_{i'})) &=\nabla f(x^*)^\top(x_{i'}^*(e_{i''}-e_{i'}))\\
&\le\|\nabla f(x^*)\|_2\|x_{i'}^*(e_{i''}-e_{i'})\|_2\\
&\le\sqrt{2}|x_{i'}^*|\left(\|\nabla f(0)\|_2+\|\nabla f(x^*)-\nabla f(0)\|_2\right)\\
&\le\sqrt{2}|x_{i'}^*|\left(\|\nabla f(0)\|_2+M\|x^*-0\|_2\right)\\
&\le\sqrt{2}|x_{i'}^*|\left(\|\nabla f(0)\|_2+M\|x^*\|_1\right)\\
&=\sqrt{2}|x_{i'}^*|\left(\|\nabla f(0)\|_2+M\right),
\end{align}
which implies that $\gamma\le\min\left\{\sqrt{2}M',\sqrt{2}\left(\|\nabla f(0)\|_2+M\right)\right\}$.
The contradiction to the assumption of $\gamma$ is derived.
\end{proof}

\subsection{Proof of Theorem \ref{thm:epp-smooth-bounded-tn}}
\begin{proof}
By the $M$-smoothness of $f$, for all $D\in\mathcal{F}(X^*;\mathcal{C})$ such that $\|D\|_*=1$, we obtain 
\begin{align}
f'(X^*,D) &=\nabla f(X^*)\bullet D\\
&\le\|\nabla f(X^*)\|_2\\
&\le\|\nabla f(0)\|_2+\|\nabla f(X^*)-\nabla f(0)\|_F\\
&\le\|\nabla f(0)\|_2+M\|X^*\|_F\\
&\le\|\nabla f(0)\|_2+MC,
\end{align}
where the first inequality follows from the generalized Cauchy-Schwarz inequality, the second one from the triangle inequality and $\|\cdot\|_2\le\|\cdot\|_F$, the third one from the $M$-smoothness of $f$, and the fourth one from the assumption $\|X^*\|_F\le C$.
Therefore, the assumption \eqref{ineq:bounded-truncated-nuclear} holds with $\Gamma=\|\nabla f(0)\|_2+MC$.
From Lemma \ref{lem:epp-GTL}, we have the desired result.
\end{proof}

\subsection{Proof of Theorem \ref{thm:epp-Lipschitz-tn}}
\begin{proof}
By the Lipschitz continuity of $f$ and Lemma \ref{lem:d-derivative-lipschitz}, we have $f'(X;D)\le M$ for all $X\in\mathcal{C}$ and $D\in\mathcal{F}(X^*;\mathcal{C})$ such that $\|D\|_*=1$.
This implies that the assumption of Lemma \ref{lem:epp-truncated-nuclear} holds with $\Gamma=M$.
This completes the proof.
\end{proof}

\subsection{Proof of Theorem \ref{thm:epp-small-rank-ols}}
\begin{proof}
First, we see that
\begin{align}
\mathcal{T}_K'(X^*;-X^*) &=\lim_{\eta\searrow0}\frac{\mathcal{T}_K(X^*-\eta X^*)-\mathcal{T}_K(X^*)}{\eta}\\
&=\lim_{\eta\searrow0}\frac{(1-\eta)\mathcal{T}_K(X^*)-\mathcal{T}_K(X^*)}{\eta}\\
&=\lim_{\eta\searrow0}\frac{-\eta\mathcal{T}_K(X^*)}{\eta}\\
&=-\mathcal{T}_K(X^*).
\end{align}
By the d-stationarity of $(x_1^*,...,x_L^*)$, we have
\begin{equation}\label{ineq:pre-descent-tn}
(\mathcal{A}^*(\mathcal{A}(X^*)-b))\bullet(0-X^*)-\gamma\mathcal{T}_K(X^*)\ge0.
\end{equation}
Since $f$ is convex, combining the inequality \eqref{ineq:pre-descent-tn}, Lemma \ref{lem:d-derivative-descent-lemma}, and non-negativity of $\mathcal{T}_K$ yields
\begin{equation}
\frac{1}{2}\|b\|_2^2\ge \frac{1}{2}\|\mathcal{A}(X^*)-b\|_2^2,
\end{equation}
that is, $\|b\|_2\ge\|\mathcal{A}(X^*)-b\|_2$.
Then, for all $D\in\mathcal{F}(X^*;\mathcal{C})$ such that $\|D\|_*=1$, we obtain
\begin{align}
f'(X^*;D) &=(\mathcal{A}^*(\mathcal{A}(X^*)-b))\bullet D\\
&\le\|\mathcal{A}^*(\mathcal{A}(X^*)-b)\|_2\\
&\le\|\mathcal{A}^*\|_{2,2}\|\mathcal{A}(X^*)-b\|_2\\
&\le\|\mathcal{A}^*\|_{2,2}\|b\|_2,
\end{align}
where the first inequality follows from the generalized Cauchy-Schwarz inequality.
From Lemma \ref{lem:epp-truncated-nuclear}, we have the desired result.
\end{proof}

\bibliography{reference.bib}

\begin{thebibliography}{52}
\providecommand{\natexlab}[1]{#1}
\providecommand{\url}[1]{\texttt{#1}}
\expandafter\ifx\csname urlstyle\endcsname\relax
  \providecommand{\doi}[1]{doi: #1}\else
  \providecommand{\doi}{doi: \begingroup \urlstyle{rm}\Url}\fi

\bibitem[Ahn et~al.(2017)Ahn, Pang, and Xin]{ahn2017difference}
Miju Ahn, Jong-Shi Pang, and Jack Xin.
\newblock Difference-of-convex learning: directional stationarity, optimality,
  and sparsity.
\newblock \emph{SIAM Journal on Optimization}, 27\penalty0 (3):\penalty0
  1637--1665, 2017.

\bibitem[Amir et~al.(2021)Amir, Basri, and Nadler]{amir2021trimmed}
Tal Amir, Ronen Basri, and Boaz Nadler.
\newblock The trimmed lasso: Sparse recovery guarantees and practical
  optimization by the generalized soft-min penalty.
\newblock \emph{SIAM Journal on Mathematics of Data Science}, 3\penalty0
  (3):\penalty0 900--929, 2021.

\bibitem[Banjac and Goulart(2017)]{banjac2017novel}
Goran Banjac and Paul~J Goulart.
\newblock A novel approach for solving convex problems with cardinality
  constraints.
\newblock \emph{IFAC-PapersOnLine}, 50\penalty0 (1):\penalty0 13182--13187,
  2017.

\bibitem[Barzilai and Borwein(1988)]{barzilai1988two}
Jonathan Barzilai and Jonathan~M Borwein.
\newblock Two-point step size gradient methods.
\newblock \emph{IMA journal of numerical analysis}, 8\penalty0 (1):\penalty0
  141--148, 1988.

\bibitem[Beck(2017)]{beck2017first}
Amir Beck.
\newblock \emph{First-order methods in optimization}.
\newblock SIAM, 2017.

\bibitem[Bertsimas et~al.(2017)Bertsimas, Copenhaver, and
  Mazumder]{bertsimas2017trimmed}
Dimitris Bertsimas, Martin~S Copenhaver, and Rahul Mazumder.
\newblock The trimmed lasso: Sparsity and robustness.
\newblock \emph{arXiv preprint arXiv:1708.04527}, 2017.

\bibitem[Bi and Pan(2016)]{bi2016error}
Shujun Bi and Shaohua Pan.
\newblock Error bounds for rank constrained optimization problems and
  applications.
\newblock \emph{Operations Research Letters}, 44\penalty0 (3):\penalty0
  336--341, 2016.

\bibitem[Chang et~al.(2017)Chang, Hong, and Pang]{chang2017local}
Tsung-Hui Chang, Mingyi~Hong Hong, and Jong-Shi Pang.
\newblock Local minimizers and second-order conditions in composite piecewise
  programming via directional derivatives.
\newblock \emph{arXiv preprint arXiv:1709.05758}, 2017.

\bibitem[Cui and Pang(2021)]{cui2021modern}
Ying Cui and Jong-Shi Pang.
\newblock \emph{Modern nonconvex nondifferentiable optimization}.
\newblock SIAM, 2021.

\bibitem[Cui et~al.(2018)Cui, Pang, and Sen]{cui2018composite}
Ying Cui, Jong-Shi Pang, and Bodhisattva Sen.
\newblock Composite difference-max programs for modern statistical estimation
  problems.
\newblock \emph{SIAM Journal on Optimization}, 28\penalty0 (4):\penalty0
  3344--3374, 2018.

\bibitem[Cui et~al.(2020)Cui, Chang, Hong, and Pang]{cui2020study}
Ying Cui, Tsung-Hui Chang, Mingyi Hong, and Jong-Shi Pang.
\newblock A study of piecewise linear-quadratic programs.
\newblock \emph{Journal of Optimization Theory and Applications}, 186:\penalty0
  523--553, 2020.

\bibitem[Delfour(2019)]{delfour2019introduction}
Michel~C Delfour.
\newblock \emph{Introduction to Optimization and Hadamard Semidifferential
  Calculus}.
\newblock SIAM, 2019.

\bibitem[Demyanov et~al.(1998)Demyanov, Di~Pillo, and
  Facchinei]{demyanov1998exact}
Vladimir~F Demyanov, Gianni Di~Pillo, and Francisco Facchinei.
\newblock Exact penalization via {Dini} and {Hadamard} conditional derivatives.
\newblock \emph{Optimization Methods and Software}, 9\penalty0 (1-3):\penalty0
  19--36, 1998.

\bibitem[Facchinei and Pang(2003)]{facchinei2003finite}
Francisco Facchinei and Jong-Shi Pang.
\newblock \emph{Finite-dimensional variational inequalities and complementarity
  problems}.
\newblock Springer, 2003.

\bibitem[Fan and Li(2001)]{fan2001variable}
Jianqing Fan and Runze Li.
\newblock Variable selection via nonconcave penalized likelihood and its oracle
  properties.
\newblock \emph{Journal of the American statistical Association}, 96\penalty0
  (456):\penalty0 1348--1360, 2001.

\bibitem[Gao and Sun(2010)]{gao2010majorized}
Yan Gao and Defeng Sun.
\newblock A majorized penalty approach for calibrating rank constrained
  correlation matrix problems.
\newblock \emph{Technical report available at
  \url{https://www.polyu.edu.hk/ama/profile/dfsun/MajorPen\_May5.pdf}}, 2010.

\bibitem[Gong et~al.(2013)Gong, Zhang, Lu, Huang, and Ye]{gong2013general}
Pinghua Gong, Changshui Zhang, Zhaosong Lu, Jianhua Huang, and Jieping Ye.
\newblock A general iterative shrinkage and thresholding algorithm for
  non-convex regularized optimization problems.
\newblock In \emph{international conference on machine learning}, pages 37--45.
  PMLR, 2013.

\bibitem[Gotoh et~al.(2018)Gotoh, Takeda, and Tono]{gotoh2018dc}
Jun-ya Gotoh, Akiko Takeda, and Katsuya Tono.
\newblock Dc formulations and algorithms for sparse optimization problems.
\newblock \emph{Mathematical Programming}, 169\penalty0 (1):\penalty0 141--176,
  2018.

\bibitem[Grippo et~al.(1986)Grippo, Lampariello, and
  Lucidi]{grippo1986nonmonotone}
Luigi Grippo, Francesco Lampariello, and Stephano Lucidi.
\newblock A nonmonotone line search technique for newton’s method.
\newblock \emph{SIAM journal on Numerical Analysis}, 23\penalty0 (4):\penalty0
  707--716, 1986.

\bibitem[Hempel and Goulart(2014)]{hempel2014novel}
Andreas~B Hempel and Paul~J Goulart.
\newblock A novel method for modelling cardinality and rank constraints.
\newblock In \emph{53rd IEEE Conference on Decision and Control}, pages
  4322--4327. IEEE, 2014.

\bibitem[Hong et~al.(2016)Hong, Wei, Hu, Cai, and He]{hong2016online}
Bin Hong, Long Wei, Yao Hu, Deng Cai, and Xiaofei He.
\newblock Online robust principal component analysis via truncated nuclear norm
  regularization.
\newblock \emph{Neurocomputing}, 175:\penalty0 216--222, 2016.

\bibitem[Hu et~al.(2012)Hu, Zhang, Ye, Li, and He]{hu2012fast}
Yao Hu, Debing Zhang, Jieping Ye, Xuelong Li, and Xiaofei He.
\newblock Fast and accurate matrix completion via truncated nuclear norm
  regularization.
\newblock \emph{IEEE transactions on pattern analysis and machine
  intelligence}, 35\penalty0 (9):\penalty0 2117--2130, 2012.

\bibitem[Hu et~al.(2015)Hu, Jin, Shi, Zhang, Cai, and He]{hu2015large}
Yao Hu, Zhongming Jin, Yi~Shi, Debing Zhang, Deng Cai, and Xiaofei He.
\newblock Large scale multi-class classification with truncated nuclear norm
  regularization.
\newblock \emph{Neurocomputing}, 148:\penalty0 310--317, 2015.

\bibitem[Huang et~al.(2015)Huang, Liu, Shi, Van~Huffel, and
  Suykens]{huang2015two}
Xiaolin Huang, Yipeng Liu, Lei Shi, Sabine Van~Huffel, and Johan~AK Suykens.
\newblock Two-level $\ell_1$ minimization for compressed sensing.
\newblock \emph{Signal Processing}, 108:\penalty0 459--475, 2015.

\bibitem[Jaggi and Sulovsk{\`y}(2010)]{jaggi2010simple}
Martin Jaggi and Marek Sulovsk{\`y}.
\newblock A simple algorithm for nuclear norm regularized problems.
\newblock In \emph{international conference on machine learning}, pages
  471--478. PMLR, 2010.

\bibitem[Kim et~al.(2009)Kim, Koh, Boyd, and Gorinevsky]{kim2009ell_1}
Seung-Jean Kim, Kwangmoo Koh, Stephen Boyd, and Dimitry Gorinevsky.
\newblock $\ell_1$ trend filtering.
\newblock \emph{SIAM review}, 51\penalty0 (2):\penalty0 339--360, 2009.

\bibitem[Kulkarni et~al.(1999)Kulkarni, Schmidt, and
  Tsui]{kulkarni1999eigenvalues}
Devadatta Kulkarni, Darrell Schmidt, and Sze-Kai Tsui.
\newblock Eigenvalues of tridiagonal pseudo-toeplitz matrices.
\newblock \emph{Linear Algebra and its Applications}, 297:\penalty0 63--80,
  1999.

\bibitem[Lewis and Sendov(2005)]{lewis2005nonsmooth}
Adrian~S Lewis and Hristo~S Sendov.
\newblock Nonsmooth analysis of singular values. part ii: Applications.
\newblock \emph{Set-Valued Analysis}, 13\penalty0 (3):\penalty0 243--264, 2005.

\bibitem[Li and Pong(2015)]{li2015global}
Guoyin Li and Ting~Kei Pong.
\newblock Global convergence of splitting methods for nonconvex composite
  optimization.
\newblock \emph{SIAM Journal on Optimization}, 25\penalty0 (4):\penalty0
  2434--2460, 2015.

\bibitem[Liu et~al.(2020)Liu, Lu, Chen, and Dai]{liu2020exact}
Tianxiang Liu, Zhaosong Lu, Xiaojun Chen, and Yu-Hong Dai.
\newblock An exact penalty method for semidefinite-box-constrained low-rank
  matrix optimization problems.
\newblock \emph{IMA Journal of Numerical Analysis}, 40\penalty0 (1):\penalty0
  563--586, 2020.

\bibitem[Lu and Li(2018)]{lu2018sparse}
Zhaosong Lu and Xiaorui Li.
\newblock Sparse recovery via partial regularization: models, theory, and
  algorithms.
\newblock \emph{Mathematics of Operations Research}, 43\penalty0 (4):\penalty0
  1290--1316, 2018.

\bibitem[Lu and Zhou(2019)]{lu2019nonmonotone}
Zhaosong Lu and Zirui Zhou.
\newblock Nonmonotone enhanced proximal dc algorithms for a class of structured
  nonsmooth dc programming.
\newblock \emph{SIAM Journal on Optimization}, 29\penalty0 (4):\penalty0
  2725--2752, 2019.

\bibitem[Lu et~al.(2019)Lu, Zhou, and Sun]{lu2019enhanced}
Zhaosong Lu, Zirui Zhou, and Zhe Sun.
\newblock Enhanced proximal dc algorithms with extrapolation for a class of
  structured nonsmooth dc minimization.
\newblock \emph{Mathematical Programming}, 176\penalty0 (1):\penalty0 369--401,
  2019.

\bibitem[Lu et~al.(2023)Lu, Li, and Xiang]{lu2023exact}
Zhaosong Lu, Xiaorui Li, and Shuhuang Xiang.
\newblock Exact penalization for cardinality and rank-constrained optimization
  problems via partial regularization.
\newblock \emph{Optimization Methods and Software}, 38\penalty0 (2):\penalty0
  412--433, 2023.

\bibitem[Luo et~al.(2013)Luo, Wang, and Zhang]{luo2013new}
Ziyan Luo, Yingnan Wang, and Xianglilan Zhang.
\newblock New improved penalty methods for sparse reconstruction based on
  difference of two norms.
\newblock \emph{Technical report}, 2013.
\newblock \doi{10.13140/RG.2.1.3256.3369}.

\bibitem[Nakayama and Gotoh(2021)]{nakayama2021superiority}
Shummin Nakayama and Jun-ya Gotoh.
\newblock On the superiority of {PGM}s to {PDCA}s in nonsmooth nonconvex sparse
  regression.
\newblock \emph{Optimization Letters}, 15\penalty0 (8):\penalty0 2831--2860,
  2021.

\bibitem[Oh et~al.(2015)Oh, Tai, Bazin, Kim, and Kweon]{oh2015partial}
Tae-Hyun Oh, Yu-Wing Tai, Jean-Charles Bazin, Hyeongwoo Kim, and In~So Kweon.
\newblock Partial sum minimization of singular values in robust pca: Algorithm
  and applications.
\newblock \emph{IEEE transactions on pattern analysis and machine
  intelligence}, 38\penalty0 (4):\penalty0 744--758, 2015.

\bibitem[Pang et~al.(2017)Pang, Razaviyayn, and Alvarado]{pang2017computing}
Jong-Shi Pang, Meisam Razaviyayn, and Alberth Alvarado.
\newblock Computing b-stationary points of nonsmooth dc programs.
\newblock \emph{Mathematics of Operations Research}, 42\penalty0 (1):\penalty0
  95--118, 2017.

\bibitem[Qian et~al.(2023)Qian, Pan, and Liu]{qian2023calmness}
Yitian Qian, Shaohua Pan, and Yulan Liu.
\newblock Calmness of partial perturbation to composite rank constraint systems
  and its applications.
\newblock \emph{Journal of Global Optimization}, 85\penalty0 (4):\penalty0
  867--889, 2023.

\bibitem[Razaviyayn et~al.(2013)Razaviyayn, Hong, and
  Luo]{razaviyayn2013unified}
Meisam Razaviyayn, Mingyi Hong, and Zhi-Quan Luo.
\newblock A unified convergence analysis of block successive minimization
  methods for nonsmooth optimization.
\newblock \emph{SIAM Journal on Optimization}, 23\penalty0 (2):\penalty0
  1126--1153, 2013.

\bibitem[Rockafellar and Wets(2009)]{rockafellar2009variational}
R~Tyrrell Rockafellar and Roger J-B Wets.
\newblock \emph{Variational analysis}, volume 317.
\newblock Springer Science \& Business Media, 2009.

\bibitem[She(2010)]{she2010sparse}
Yiyuan She.
\newblock Sparse regression with exact clustering.
\newblock \emph{Electronic Journal of Statistics}, 4:\penalty0 1055--1096,
  2010.

\bibitem[Tibshirani(1996)]{tibshirani1996regression}
Robert Tibshirani.
\newblock Regression shrinkage and selection via the lasso.
\newblock \emph{Journal of the Royal Statistical Society: Series B
  (Methodological)}, 58\penalty0 (1):\penalty0 267--288, 1996.

\bibitem[Tibshirani and Taylor(2011)]{tibshirani2011solution}
Ryan~J Tibshirani and Jonathan Taylor.
\newblock The solution path of the generalized lasso.
\newblock \emph{The Annals of Statistics}, 39\penalty0 (3):\penalty0
  1335--1371, 2011.

\bibitem[Tono et~al.(2017)Tono, Takeda, and Gotoh]{tono2017efficient}
Katsuya Tono, Akiko Takeda, and Jun-ya Gotoh.
\newblock Efficient dc algorithm for constrained sparse optimization.
\newblock \emph{arXiv preprint arXiv:1701.08498}, 2017.

\bibitem[Wang et~al.(2021)Wang, Chang, Cui, and Pang]{wang2021clustering}
Shuai Wang, Tsung-Hui Chang, Ying Cui, and Jong-Shi Pang.
\newblock Clustering by orthogonal nmf model and non-convex penalty
  optimization.
\newblock \emph{IEEE Transactions on Signal Processing}, 69:\penalty0
  5273--5288, 2021.

\bibitem[Wright et~al.(2009)Wright, Nowak, and Figueiredo]{wright2009sparse}
Stephen~J Wright, Robert~D Nowak, and M{\'a}rio~AT Figueiredo.
\newblock Sparse reconstruction by separable approximation.
\newblock \emph{IEEE Transactions on signal processing}, 57\penalty0
  (7):\penalty0 2479--2493, 2009.

\bibitem[Yagishita and Gotoh(2024)]{yagishita2024pursuit}
Shotaro Yagishita and Jun-ya Gotoh.
\newblock Pursuit of the cluster structure of network lasso: Recovery condition
  and non-convex extension.
\newblock \emph{Journal of Machine Learning Research}, 25\penalty0
  (21):\penalty0 1--42, 2024.

\bibitem[Yun et~al.(2019)Yun, Zheng, Yang, Lozano, and
  Aravkin]{yun2019trimming}
Jihun Yun, Peng Zheng, Eunho Yang, Aurelie Lozano, and Aleksandr Aravkin.
\newblock Trimming the $\ell_1$ regularizer: Statistical analysis,
  optimization, and applications to deep learning.
\newblock In \emph{International Conference on Machine Learning}, pages
  7242--7251. PMLR, 2019.

\bibitem[Zhang(2010)]{zhang2010nearly}
Cun-Hui Zhang.
\newblock Nearly unbiased variable selection under minimax concave penalty.
\newblock \emph{The Annals of Statistics}, 38\penalty0 (2):\penalty0 894--942,
  2010.

\bibitem[Zhang et~al.(2012)Zhang, Hu, Ye, Li, and He]{zhang2012matrix}
Debing Zhang, Yao Hu, Jieping Ye, Xuelong Li, and Xiaofei He.
\newblock Matrix completion by truncated nuclear norm regularization.
\newblock In \emph{2012 IEEE Conference on computer vision and pattern
  recognition}, pages 2192--2199. IEEE, 2012.

\bibitem[Zhang et~al.(2016)Zhang, Guo, Zhao, and Wang]{zhang2016robust}
Yan Zhang, Jichang Guo, Jie Zhao, and Bo~Wang.
\newblock Robust principal component analysis via truncated nuclear norm
  minimization.
\newblock \emph{Journal of Shanghai Jiaotong University (Science)}, 21\penalty0
  (5):\penalty0 576--583, 2016.

\end{thebibliography}
\bibliographystyle{plainnat}

\end{document}